\documentclass{amsart}
\usepackage[usenames,dvipsnames]{color}
\usepackage[dvips]{graphicx}
\usepackage{dsfont}
\usepackage{hyperref}
\usepackage[all]{xy}
\usepackage{multirow, colortbl}
\usepackage{amsmath}
\usepackage{amssymb}

\renewcommand{\epsilon}{\varepsilon}

\definecolor{GGray}{gray}{0.8}

\newcommand{\N}{{\mathbb N}}
\newcommand{\Z}{{\mathbb Z}}

\newcommand{\R}{{\mathbb R}}
\newcommand{\C}{{\mathbb C}}

\newcommand{\eff}{\operatorname{eff}}
\newcommand{\Vol}{\operatorname{Vol}}

\newcommand{\card}{\operatorname{card}}
\newcommand{\SL}{\operatorname{SL}(2,\mathbb{R})}
\newcommand{\cC}{{\mathcal C}}
\newcommand{\cT}{{\mathcal T}}
\newcommand{\cH}{{\mathcal H}}
\newcommand{\cQ}{{\mathcal Q}}
\newcommand{\dd}{\mathrm{d}}
\renewcommand{\epsilon}{\varepsilon}
\newtheorem{theorem}{Theorem}
\newtheorem{lemma}{Lemma}
\newtheorem{corollary}{Corollary}
\newtheorem{proposition}{Proposition}
\theoremstyle{remark}
\newtheorem{remark}{Remark}
\newtheorem{example}{Example}
\newtheorem{conv}{Convention}

\begin{document}

\date{}

\title[Siegel--Veech constants]{Siegel--Veech constants for strata of moduli spaces of quadratic differentials}

\author[E. Goujard]{Elise Goujard}

\begin{abstract}We present an explicit formula relating volumes of strata of meromorphic quadratic differentials with at most simple poles on Riemann surfaces and counting functions of the number of flat cylinders filled by closed geodesics in associated  flat metric with singularities. This generalizes the result of Athreya, Eskin and Zorich in genus 0 to higher genera.\end{abstract}

\maketitle

\section{Introduction}

\subsection{Cylinders and saddle connections on half-translation surfaces}

A meromorphic quadratic differential $q$ with at most simple poles on a Riemann surface $S$ of genus $g$ defines a flat metric on $S$ with conical singularities. If $q$ is \emph{not} the global square of a holomorphic 1-form on $S$, the metric has a non-trivial linear holonomy group, and in this case $(S,q)$ is called a {\it half-translation} surface. In this paper we consider only quadratic differentials satisfying the previous condition. If $\alpha=\{\alpha_1, \dots, \alpha_n\}\subset\{-1\}\cup\N$ is a partition of $4g-4$, $\cQ(\alpha)$ denotes the moduli space of pairs $(S,q)$ as above, where $q$ has exactly $n$ singularities of orders given by $\alpha$. It is a {\it stratum} in the moduli space $\cQ_g$ of pairs $(S,q)$ with no additional constraints on $q$.

In what follows we will refer to a half-translation surface $(S,q)$ simply as $S$. 

A {\it saddle connection} on $S$ is a geodesic segment on $S$ joining a pair of conical singularities or a singularity to itself without any singularities in its interior. Note that maximal flat cylinders filled by parallel regular closed geodesics have their boundaries composed by one or several parallel saddle connections. In this paper we will evaluate the number of such cylinders on $S$ in terms of the volumes of some strata, using the study of saddle connections by Masur and Zorich in \cite{MZ}.

\subsection{Rigid collections of saddle connections}
A saddle connection persists under any small deformation of $S$ inside the stratum $\cQ(\alpha)$. Moreover Masur and Zorich noticed in \cite{MZ} that in some cases any small deformation which shortens a specific saddle connection shortens also some other saddle connections.  More precisely, they give the following result (Proposition 1 of \cite{MZ}):
\begin{proposition}[Masur-Zorich]
Let $\{\gamma_1, \dots, \gamma_m\}$ be a collection of saddle connections on a half-translation surface $S$. Then any sufficiently small deformation of $S$ inside the stratum preserves the proportions $\vert\gamma_1\vert:\vert\gamma_2\vert:\dots:\vert\gamma_m\vert$ of the lengths of the saddle connections if and only if the saddle connections are \^homologous.
\end{proposition}

 Roughly two saddle connections are \^homologous if they define the same anti-invariant cycle in the orientation double cover. The precise definition will be recalled in \S~\ref{ssect:hom}.
In particular two \^homologous saddle connections are parallel with ratios of lengths equal to $1$ or $2$.

The geometric types of possible maximal collections of \^homologous saddle connections $\gamma=\{\gamma_1,\dots, \gamma_m\}$ on $S$ are called {\it configurations} of saddle connections. Masur and Zorich classified 
all configurations of saddle connections in \cite{MZ} in terms of combinatorial data.

We assume in the sequel that $S$ belongs to a connected stratum (unless the non connectedness is stated explicitly), and we will not make a distinction when we speak about configurations for the surface $S$ or for the stratum $\cQ(\alpha)$, the second means that we look at all possible configurations on almost every surface $S\in\cQ(\alpha)$.

We are interested in collections of \^homologous saddle connections, such that some of the saddle connections bound at least one cylinder filled by parallel regular closed geodesics. We refer to the geometric type of these collections as ``configurations containing cylinders'' or ``configurations with cylinders''.

It is proved in \cite{MZ} that such cylinders have in fact each of their two boundaries composed by exactly one or two saddle connections in the collection, and that if there are several cylinders in the configuration, the lengths of their waist curves are either the same or have the ratio 1:2. Namely, some cylinders have their width twice larger than the width of the other cylinders. The boundary of the first cylinders are composed either by one or two saddle connections, and the boundary of the latter cylinders are composed by exactly one saddle connection. We will refer to cylinders of the first type as ``thick cylinders'' and to cylinders of the second type as ``thin cylinders''. We call the length of the minimal saddle connection in the collection or equivalently the width of any thin cylinder the ``length of the configuration''.

Let $\gamma$ be a maximal collection of \^homologous saddle connections on $S$. Then the complimentary region of these saddle connections and the cylinders bounded by these saddle connections is the union of some surfaces with boundaries. Each of them might be obtained by a specific surgery from a flat surface belonging to a stratum $\cQ(\alpha_i)$ or $\cH(\beta_j)$. The union of these strata $\cQ(\alpha')=\cup_{i,j}\cQ(\alpha_i)\cup\cH(\beta_j)$ is called the {\em boundary stratum} for the configuration $\cC$. This distinction is meaningful: the boundary stratum corresponds to the degeneration of the stratum $ \cQ(\alpha) $ as the lengths of the saddle connections in the collection tend to 0. 

\subsection{Counting saddle connections}

Let $S$ be a half-translation surface in a connected stratum $\cQ(\alpha)$, and $\cC$ a configuration with cylinders on $S$. It means that in some given direction, there is a collection of \^homologous saddle connections of type $\cC$ on $S$. Note that by results of \cite{EM} in many other directions, one can usually find another collection of \^homologous saddle connections of same type $\cC$.

 We introduce $N(S, \cC, L)$ the number of directions on $S$ in which we can find a collection of saddle connections of type $\cC$, with the length of the smallest saddle connection smaller than $L$.
Since we are interested in cylinders we introduce also $N_{cyl}(S, \cC, L)$ that counts each appearance of the configuration $\cC$ with weight equal to the number of the cylinders of width smaller than $ L $, and $N_{area}(S, \cC, L)$ that counts each appearance of the configuration $\cC$ with weight equal to the area of the cylinders of width smaller than $ L $.

For each of these numbers, we introduce the corresponding Siegel--Veech constant, that gives the asymptotic of these numbers as $L$ goes to infinity:
\begin{equation}\label{eq:defSV} c_{*}(\cC)=\lim\limits_{L\to\infty}\cfrac{N_{*}(S, \cC, L)\cdot(\mbox{Area of }S)}{\pi L^2}.\end{equation}
The stratum $\cQ(\alpha)$ is equipped with a natural $PSL(2, \R)$-invariant measure induced by the Lebesgue measure in period coordinates.
Eskin and Masur showed in \cite{EM} that the numbers \eqref{eq:defSV} are in fact constant on a full measure set of the stratum $\cQ(\alpha)$. Combining these results with the results of Veech \cite{Ve}, one concludes that all these constants are strictly positive.

\subsection{Application of Siegel--Veech constants}
One of the principal reasons why the Siegel--Veech constants are more and more intensively studied during the last years \cite{AEZ,Bainbridge1,Bainbridge2,BG,EKZ,Vo} is the relation between them and the Lyapunov exponents of the Hodge bundle along the Teichm\"uller flow: the key formula of~\cite{EKZ} expresses the sum of the positive Lyapunov exponents for any stratum $\cQ(\alpha)$ as a sum of a very explicit rational function in $\alpha$ and the Siegel--Veech constant $c_{area}(\cQ(\alpha))$. The Lyapunov exponents are closely related to the deviation spectrum of measured foliations on individual flat surfaces \cite{Forni:Deviation,Forni:Handbook,Zorich:Asymptotic:flag,Zorich:How:do}, which opens applications to billiards in polygons, interval exchanges, etc.

A recent breakthrough of A.~Eskin and M.~Mirzakhani provides, in particular, new tools allowing to prove that the $\SL$-orbit closure of certain individual flat surfaces is an entire stratum. By the theorem of J.~Chaika and A.~Eskin~\cite{Chaika:Eskin}, almost all directions for such a flat surface are Lyapunov-generic. This allows to use all the technology mentioned above to compute, for example, the diffusion rate of billiards with certain periodic obstacles. The final explicit answer (as $2/3$ for the diffusion rate in the windtree model studied in~\cite{Delecroix:Hubert:Lelievre}) is a certain Lyapunov exponent as above. These kinds of quantitative answers or estimates are often reduced to computation of the appropriate Siegel--Veech constants.

The Kontsevich formula \cite{Kontsevich} for the sum of the Lyapunov exponents over a Teichm\"uller curve  and recent results of S.~Filip~\cite{Filip} showing that every orbit closure is a quasiprojective variety suggest that an adequate intersection theory of the strata might provide algebro-geometric tools to evaluate Siegel--Veech constants (see also~\cite{Korotkin:Zograf} in this connection). However, such intersection theory is not developed yet, and we are limited to analytic tools in our evaluation of Siegel--Veech constants.

\subsection{Principal results}

Now we are ready to state the main theorem of this paper.
\begin{theorem}\label{th:ccyl} Let $\cC$ be an admissible configuration for a connected stratum $\cQ(\alpha)$ of quadratic differentials. Let $q_1$ denote the number of thin cylinders, $q_2$ the number of thick cylinders in the configuration $\cC$, and $q=q_1+q_2$ the total number of cylinders. Assume that the boundary stratum $\cQ(\alpha')$ is non empty, and $q\geq 1$. Then the Siegel--Veech constants associated to $\cC$ are the following:
\begin{eqnarray}c(\cC) &= &\frac{M}{2^{q+2}}\frac{(\dim_{\C} \cQ(\alpha')-1)!}{(\dim_{\C} \cQ(\alpha)-2)!}\frac{\Vol \cQ_1(\alpha')}{\Vol \cQ_1(\alpha)}\label{eq:cSV}\\
c_{cyl}(\cC)& =&\left(q_1+\frac{1}{4}q_2\right)c(\cC)\label{eq:ccylc}\\
c_{area}(\cC)&=&\frac{1}{\dim_{\C} \cQ(\alpha)-1}c_{cyl}(\cC)\label{eq:careaandc}\end{eqnarray}
where $M=\cfrac{M_sM_c}{M_t}$ and $M_c$, $M_t$, $M_s$ are combinatorial constants depending only on the configuration $\cC$, explicitly given by equations (\ref{eq:M_c}), (\ref{eq:M_t}) and (\ref{eq:M_s}).
\end{theorem}
When the boundary stratum is empty, the formulas are simpler and given in \S~\ref{sssection:specialcase}.

This theorem is proved in section \ref{ssection:computation}. Note that these formulas coincide in genus 0 with the formulas of \cite{AEZ}, for the two configurations containing cylinders (named ``pocket'' and ``dumbell'' in the article).

The ratio $\cfrac{c_{area}(\cC)}{c_{cyl}(\cC)}=\cfrac{1}{\dim_\C\cQ(\alpha)-1}$ can be interpreted as the mean area of a cylinder in the configuration $\cC$. Note that it depends only on the dimension of the ambient stratum.

For a fixed stratum $\cQ(\alpha)$ consider all admissible configurations, and denote $q_{max}(\alpha)$ the maximal number of cylinders for all these configurations. We evaluate this number in section \ref{ssect:maxcyl}. The ratio $\cfrac{q_{max}(\alpha)}{\dim_{\C}(\cQ(\alpha))-1}$ represents the maximum mean total area of the cylinders in stratum $\cQ(\alpha)$. 
\begin{proposition}
\label{prop:qmax:dim}
We have
$$
\begin{array}{cc}\underset{\alpha\in\Pi(4g-4+k)}{\max}\ \cfrac{ q_{\max}(\alpha\cup\{-1^k\})} {2g-3+\ell(\alpha)+k}  & \xrightarrow[g\to\infty]{k \,\textrm{ fixed}} \cfrac{1}{3}\\ & \xrightarrow[k\to\infty]{g \,\textrm{ fixed}} \cfrac{1}{5}\end{array}
$$
where $\Pi (4g-4+k)$ denotes the set of partitions of $4g-4+k$ and $l(\alpha)$ is the length of the partition $\alpha$.
Furthermore for any genus $g$ and number of poles $k$ the bound is achieved for $\alpha\in\Pi (k')\sqcup \Pi_4(4g-4+k-k')$, where $k'=k-4\left\lfloor\cfrac{k}{4}\right\rfloor$ and $\Pi_{4}(4g-4+k-k')$ denote the set of partitions of $4g-4+k-k'$ using only $4$'s. 
\end{proposition}

\subsection{Historical remarks}
The Siegel--Veech constants for the strata of Abelian differentials were evaluated in the paper~\cite{EMZ}; the relations between various Siegel--Veech constants were studied in~\cite{Vo} and some further ones in a recent paper~\cite{BG}. The computation in~\cite{EMZ} involves a combination of rather involved combinatorial and geometric constructions. To test the consistency of their theoretical predictions numerically, the authors of~\cite{EMZ} compare the formulas for the Lyapunov exponents expressed in terms of the Siegel--Veech constants (reduced, in turn, to combinations of volumes of the boundary strata) with numerics provided by experiments with the Lyapunov exponents. These tests are based, in particular, on the results of A.~Eskin and A.~Okounkov~\cite{EO} providing the explicit values of the volumes of all strata of Abelian differentials in small genera.

The description of combinatorial geometry of configurations of saddle connections for the strata of quadratic differentials is performed in the paper of H.~Masur and A.~Zorich~\cite{MZ}; for the hyperelliptic components and for strata in genus zero such description is given in the paper of C.~Boissy~\cite{B}.

The evaluation of the corresponding Siegel--Veech constants in genus zero was recently performed by J.~Athreya, A.~Eskin, and A.~Zorich~\cite{AEZ}; see also the related paper~\cite{AEZ2}. The results were also verified by computer experiments with Lyapunov exponents combined with the knowledge of the volumes of the strata of quadratic differentials in genus zero. (The authors prove in~\cite{AEZ} an extremely simple explicit formula for such volumes in genus zero conjectured by M.~Kontsevich.)

In the current paper we treat the strata of quadratic differentials in arbitrary genus. We should point that, in the contrast to the strata of Abelian differentials, the analogous results of A.~Eskin, A.~Okounkov~\cite{EO2}, and R.~Pandharipande~\cite{Eskin:Okounkov:Pandharipande} do not provide explicit values for the volumes of the strata of quadratic differentials. This is why in \cite{G} we have computed the values of volumes of a large amount of strata in low dimension, implementing the algorithm of Eskin and Okounkov. These volumes were independently tested in \cite{DGZZ}. In this paper we use these values to obtain some exact values of Siegel--Veech constants for the strata of quadratic differentials away from genus zero, and to show that our formulas for Siegel--Veech constants are consistent with numerics coming from Lyapunov exponents of the Hodge bundle over the Teichm\"uller flow. Furthermore, we have compared all our results with the program POLYGON of Alex Eskin, which counts configurations of saddle connections for individual translation surfaces.

\subsection{Structure of the paper}
The paper is divided into five parts. The first two sections, theoretical, give the proof of Theorem \ref{th:ccyl}, and develop the results on a special family of strata: $\cQ(1^k, -1^l)$. For the first strata of this type we compute the exact values of Siegel--Veech constants and we obtain the exact values of the sums of Lyapunov exponents of the Hodge bundle along the Teichm\"uller flow. These values are important for applications to billiards, such as windtree models \cite{Delecroix:Hubert:Lelievre}.

The computations of this first part generalize the computations presented in the articles \cite{EMZ}, and \cite{AEZ}, but in higher genus, this theory does not easily yield the exact values of Siegel-Veech constants, because the techniques involve phenomena of higher complexity. This is why we present in a second part all explicit computations.

Section \ref{sect:volhyp} develops the formula in the case of hyperelliptic components of strata. For these components, the values of the volumes and the Siegel--Veech constants are known, which enables us to check the coherence of the formulas in this case.

Section \ref{sect:ex} is devoted to the application of the main formula for strata of small dimension where we have explicit values of the volumes \cite{G}. In particular for non-varying strata we check the coherence of the main formula. In the other cases we get new explicit Siegel--Veech constants that we can compare with the experimental value of the sum of Lyapunov exponents. This comparison serves as an independent test of coherence of our choice of numerous normalizations and a confirmation that all discrete symmetries of relatively sophisticated configurations are taken into consideration.

We complete the paper with the extension of some geometric results proved in \cite{BG} for the strata of Abelian differentials to the strata of quadratic differentials.

\subsection{Acknowledgements}
I wish to thank my advisor Anton~Zorich, for his guidance and support during the preparation of this paper. I am grateful to Alex~Eskin to letting me use his program on configurations to check the computations of this paper and for his enlightening explanations of the method of volume computations. I thank Pascal~Hubert for helpful remarks on the preliminary version of this paper, Max~Bauer for many helpful discussions related to Siegel--Veech constants. I thank Anton Zorich and Charles Fougeron for providing me numerical data on Lyapunov exponents. I am grateful to Howard Masur and Anton Zorich for letting me use their pictures from \cite{MZ}. I thank the anonymous referee for useful comments and careful reading of the manuscript. I thank ANR GeoDyM for financial support.

\section{Preliminaries }\label{sect:pre}
\subsection{\^Homologous saddle connections}\label{ssect:hom}
We recall here from \cite{MZ} the notion of {\em \^homologous} saddle connections. 

 Any flat surface $(S,q)$ in $\cQ(\alpha)$ admits a canonical ramified double cover $\hat{S}\overset{p}{\rightarrow} S$ such that the induced quadratic differential on $\hat{S}$ is a global square of an Abelian differential, that is $p^* q= \omega^2$ and $(\hat S, \omega)\in \cH(\hat \alpha)$. Let $\Sigma=\{P_1, \dots P_n\}$ denote the singular points of the quadratic differential on $S$, and $\hat\Sigma=\{\hat P_1, \dots \hat P_N\}$ the singular points of the Abelian differential $\omega$ on $\hat S$. Note that the pre-images of poles $P_i$ are regular points of $\omega$ so do not appear in the list $\hat\Sigma$. The subspace $H^1_-(\hat S, \hat\Sigma; \C)$ anti-invariant with respect to the action of the hyperelliptic involution provides local coordinates in the stratum $\cQ(\alpha)$ in the neighborhood of $S$.
 
 Let $ \gamma $ be a saddle connection on $ S $. We denote $\gamma'$ and $\gamma''$ its two lifts on $\hat S$. If $[\gamma]= 0$ downstairs, then $[\gamma']+[\gamma'']= 0$ upstairs, and in this case we define $[\hat \gamma]   := [\gamma']$. In the other case we have $[\gamma']+[\gamma'']\neq 0$ and we define $[\hat \gamma]:=[\gamma']-[\gamma'']$. We obtain
  an element of $H^1_-(\hat S, \hat\Sigma; \C)$.
  
 Then two saddle connections $ \gamma_1 $ and $ \gamma_2 $ are said to be \^homologous if $ [\hat\gamma_1]=[\hat\gamma_2] $ in $H_1(\hat S, \hat\Sigma, \Z)$, under an appropriate choice of orientations of $\gamma_1, \gamma_2$.

 \subsection{Configurations of saddle connections}
 A configuration is one of the geometric type of all possible maximal collections of \^homologous saddle connections. We detail here precisely the information that characterizes the geometric type of a collection (Definition 3 of \cite{MZ}).
 Given such a collection of saddle connections on a surface $S$, cutting along these saddle connections gives a union of surfaces with boundaries. These surfaces can be either flat cylinders, or surfaces obtained by a surgery from a surface of trivial or non trivial holonomy. These surfaces are called boundary surfaces. We record the genus and the order of the singularities of all these surfaces. We record also which type of surgery is applied to which singularity on each surface with the precise angles. Finally we record the way the surfaces are glued in the initial surface. All this information characterizes a configuration of \^homologous saddle connections. 

\subsection{Graphs of configurations }\label{ssect:graph}

We recall here briefly how the graphs introduced by Masur and Zorich in \cite{MZ} encode all combinatorial information about a configuration. We reproduce Figure 3 and Figure 6 of \cite{MZ} describing the graphs on Figure \ref{fig:fig3MZ} and Figure \ref{fig:fig6MZ} in order to keep the paper self-contained. For a complete description of the configurations using graphs see the original article \cite{MZ}. 

\begin{figure}
\centering
\includegraphics[scale=1]{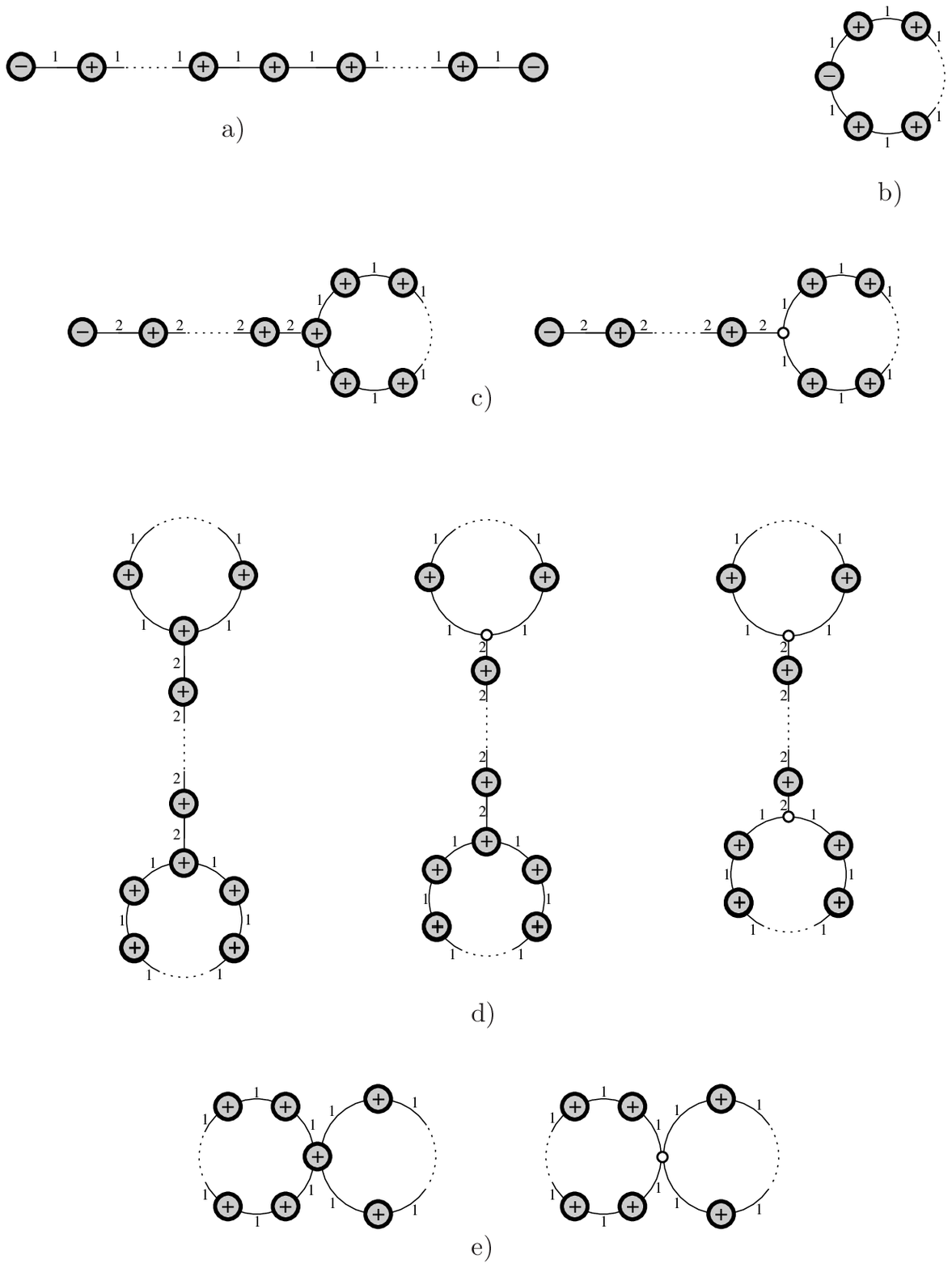}
\caption{
\label{fig:fig3MZ}
Figure 3 of \cite{MZ}: Classification of admissible graphs.
}
\end{figure}

\begin{figure}
\centering
\includegraphics[scale=1]{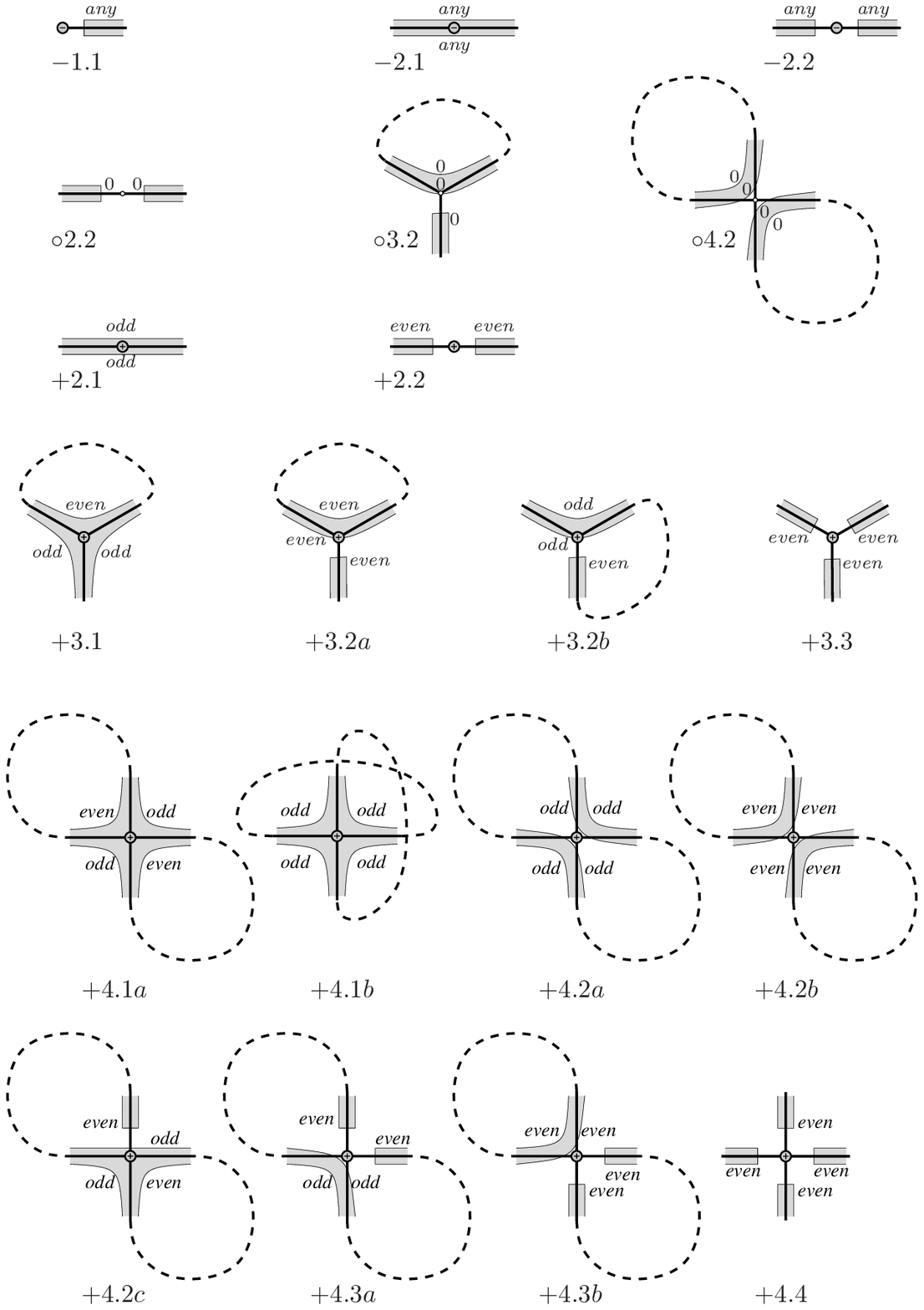}
\caption{
\label{fig:fig6MZ}
Figure 6 of \cite{MZ}: Classification of embedded local ribbon graphs.
}
\end{figure}

Let $S$ be a half-translation and $\gamma$ a saddle connection of configuration $\cC$. The graph of the configuration $\cC$ is given by the following procedure: associate to each boundary surface a vertex in the graph, with the following symbolic: a vertex $\oplus$ represents a surface of trivial holonomy, a vertex $\ominus$ a surface of non trivial holonomy, and a vertex $\circ$ a cylinder. There is an edge between two vertices if the boundaries of the corresponding surfaces share a common saddle connection. At this stage we obtain a graph described by Figure \ref{fig:fig3MZ}. 

The surgeries performed on each surface are represented by local ribbon graphs belonging to the list described in Figure \ref{fig:fig6MZ}. These local graphs are decorated with numbers $k_i$ which are the numbers of horizontal geodesic rays emerging from the zeros on which we perform the surgery, in an angular sector delimited by two \^homologous saddle connections.  The union of these local ribbon graphs forms globally a ribbon graph that can be drawn on the graph giving the organization of the surfaces. The boundary of this ribbon graph has several connected components, each of them represents a newborn zero. To compute the order of a newborn zero, one can count the number of geodesic rays emerging from this point, that is, sum all the $k_i$'s met when one goes along the  connected component of the boundary of the ribbon graph corresponding to the newborn zero. The cone angle around this point is then $\pi\sum_i (k_i+1)$. See Figure 7 in \cite{MZ} for an example. 

\subsection{General strategy for the computation of Siegel--Veech constants}\label{ssect:genstrat}
We recall here the sketch of the general method developed in \cite{EMZ} to evaluate Siegel--Veech constants in the Abelian case, transposed to the quadratic case in genus 0 in \cite{AEZ}.

Let $ V_{\cC}(S) $ be the set of holonomy vectors of saddle connections on $ S $ of type $ \cC $. The number of configurations $\cC$ in $ S $ such that the length of the \^homologous saddle connections is bounded is then \[ N(S, \cC, L)=\frac{1}{2}\vert V_{\cC}(S)\cap B(0,L)\vert,\]
where the factor $ \frac{1}{2} $ compensates the fact that the saddle connections are not oriented and so their holonomy vectors are defined up to a sign.
If $ q $ is the number of cylinders in the configuration and $ q_1 $ the number of ``thin'' cylinders, we define as well \[ N_{cyl}(S, \cC, L)=\frac{1}{2}\left(q\left\vert V_{\cC}(S)\cap B\left(0,\frac{L}{2}\right)\right\vert+q_1\left\vert V_{\cC}(S)\cap A\left(\frac{L}{2},L\right)\right\vert\right),\] with $ A\left(\frac{L}{2},L\right)=B(0,L)\setminus B\left(0,\frac{L}{2}\right)$. Note that $ N_{cyl}(S, \cC, L) $ counts each realization of configuration $ \cC $ with weight the number of cylinders of width smaller than $ L $: if the width of the thin cylinders is smaller than $ L/2 $ then all the $ q $ cylinders have their width smaller than $ L $, if the width of the thin cylinders is comprised between $ L/2 $ and $ L $, then the thick cylinders do not count. 

Simplifying the last expression we get \begin{equation}\label{eq:Ncyl} N_{cyl}(S, \cC, L)=q_2 N(S, \cC, L/2) + q_1 N(S, \cC, L)\end{equation} where $ q_2 $ is the number of thick cylinders ($ q=q_1+q_2 $).

Finally we define
\[ N_{area}(S, \cC, L)=\frac{1}{2}\sum_{v\in  V_{\cC}(S)\cap B(0,L)} A(v)\]
where $ A(v) $ is the area of the cylinders of width smaller than $ L $ among those associated to the saddle connections of type $ \cC $ and holonomy vector $ \pm v$. Note that $ N_{area}(S, \cC, L) $ weights only the cylinders which are counted by $ N_{cyl}(S, \cC, L) $.

\begin{conv}\label{convarea} Following \cite{AEZ} we denote $\cQ_1(\alpha)$ the hypersurface in $\cQ(\alpha)$ of flat surfaces of area $1/2$ such that the area of the double cover is $1$.\end{conv}

Let $\mu$ denote the natural $ PSL(2, \R) $-invariant measure on $\cQ(\alpha)$, called Masur-Veech measure, induced by the Lebesgue measure in period coordinates. We choose a normalization for $ \mu $ in \S\ref{ssect:norm}. This measure induces a measure $ \mu_1 $ on $ \cQ_1(\alpha) $ in the following way:
if $ E $ is a subset of $ \cQ_1(\alpha) $, we denote $ C(E) $ the cone underneath $ E $ in the stratum $ \cQ(\alpha) $: \[C(E)=\{S\in\cQ(\alpha)\mbox{ s.t. } \exists r\in (0, +\infty),\; S=rS_1 \mbox{ with } S_1\in E\} \]
and we define 
\[\mu_1(E)=2d\cdot\mu(C(E)), \]
with $ d=\dim_C\cQ(\alpha) $, that is, the measure $ \dd\mu $ disintegrates in $ \dd\mu=r^{2d-1}\dd r\dd\mu_1 $.

Eskin and Masur proved in \cite{EM} that the asymptotic \[\lim\limits_{L\to\infty}\frac{N_*(S, \cC, L)\cdot(\mbox{Area of }S)}{\pi L^2}\] does not depend on the surface $S$ for almost every surface in a connected component of a stratum of Abelian or quadratic differentials. This constant is denoted by $c_{*}(\cC)$ and it is called the Siegel--Veech constant for the configuration $\cC$.

\begin{remark}Note that it follows directly from this formula and the definition (\ref{eq:Ncyl}) of $ N_{cyl}(S, \cC, L) $ that: \[c_{cyl}(\cC)=\left(q_1+\frac{1}{4}q_2\right)c(\cC),\]
which is the equation (\ref{eq:ccylc}) in Theorem \ref{th:ccyl}.
\end{remark}

Now let $\cQ(\alpha)$ be a connected stratum.
The Siegel--Veech formula (cf \cite{Ve}, Theorem 0.5) gives the existence of constants $ b_*(\cC) $ such that 
\[\frac{1}{\Vol(\cQ_1(\alpha))}\int_{\cQ_1(\alpha)}N_*(S, \cC, L)\dd\mu_1(S)=b_*(\cC)\pi L^2\]
so necessarily $ b_*(\cC)=2c_*(\cC) $ and we can express the Siegel--Veech constant as \[c_*(\cC)=\lim\limits_{\varepsilon\to 0}\frac{1}{2\pi\varepsilon^2}\frac{1}{\Vol(\cQ_1(\alpha))}\int_{\cQ_1(\alpha)}N_*(S, \cC, \varepsilon)\dd\mu_1(S).\]
Actually the integral is over the subset $ \cQ_1^{\varepsilon}(\cC)$ of $ \cQ_1(\alpha) $ formed by the surfaces with at least one family of ``short'' saddle connections of type $ \cC $, where ``short'' means of length smaller than $ \varepsilon $. We decompose this subset as $ \cQ_1^{\varepsilon}(\cC)=\cQ_1^{\varepsilon, thick}(\cC)\cup \cQ_1^{\varepsilon, thin}(\cC) $ where $ \cQ_1^{\varepsilon, thin}(\cC) $ is the set of surfaces having at least two distinct collections of short saddle connections of type $ \cC $. Eskin and Masur proved in \cite{EM} that this subset is so small that we have \[\frac{1}{\Vol\cQ_1(\alpha)}\int_{\cQ_1^{\varepsilon, thin}(\cC)}N_*(S, \cC, \varepsilon)\dd\mu_1(S)=o(\varepsilon^2).\]

Finally we obtain 
\begin{equation}c_*(\cC)=\lim\limits_{\varepsilon\to 0}\frac{1}{2\pi\varepsilon^2}\frac{\Vol_*\cQ_1^{\varepsilon}(\cC)}{\Vol\cQ_1(\alpha)}\label{eq:SVbase}\end{equation}

where $ \Vol_*\cQ_1^{\varepsilon}(\cC) $ is the weighted volume:
\[\Vol_*\cQ_1^{\varepsilon}(\cC)=\int_{\cQ_1^{\varepsilon}(\cC)}W_*(\cC,S)\dd\mu_1(S)\]
with $ W(\cC,S)=1 $, $ W_{cyl}(\cC, S)$ is equal to the number of cylinders of width smaller than $ \varepsilon $, $ W_{area}(\cC,S) $ is equal to the area of the cylinders of length smaller than $\varepsilon$ in the configuration $ \cC $ on $ S $.

The last step is the computation of $ \Vol_*\cQ_1^{\varepsilon}(\cC) $ in term of the volume of the boundary stratum, see \S~\ref{ssection:computation}.

Counting saddle connections of type $ \cC $ is related to a more general problem: counting saddle connections with no fixed type. Introducing the number $ N(S, L) $ of distinct holonomies of saddle connections shorter than  $ L $ on $ S\in\cQ(\alpha) $, the corresponding Siegel--Veech constants \[c_*(\cQ(\alpha))=\lim\limits_{L\to\infty}\frac{N(S, L)\cdot(\mbox{Area of }S)}{\pi L^2}\] are also well-defined for almost every $ S\in\cQ(\alpha) $ and depend only of the stratum. 
Then we have naturally \[c_*(\cQ(\alpha))=\sum_{\cC}c_*(\cC).\]
The constant $ c_{area}(\cQ(\alpha)) $ is particularly important because the formula of \cite{EKZ} relates it to the sum of Lyapunov exponents for the Teichm\"uller geodesic flow. So it implies a lot of applications to the dynamics in polygonal billiards. Using the numerical experiments on Lyapunov exponents performed in particular by A. Zorich, V. Delecroix and C. Fougeron, the Eskin-Kontsevich-Zorich formula provides numerical approximation for the constants $ c_{area}(\cQ(\alpha)) $, and that gives a way to check computations on the constants $ c_{area}(\cC) $. This is the main reason why we focus on configurations containing cylinders: they are the only ones that contribute to the constant $ c_{area}(\cQ(\alpha)) $.

Note that this computation is somehow an analog to one of Mirzakhani, but in the flat world: in \cite{Mi}, Mirzakhani shows that the number of simple closed geodesics on a hyperbolic surface is asymptotically $cL^{6g-6}$, where the constant is related to the Weil--Peterson volumes; doing a similar counting for flat metric with singularities (in the same conformal class) we get $cL^2$, where the constant is also expressed in terms of the Masur--Veech volumes.

\subsection{Strata that are not connected}\label{ssect:noncon}
In the last section we explained the method to compute Siegel--Veech constants for connected strata. The classification of connected components of strata is given in \cite{L2}. Most of the strata are connected, the only ones that are not connected are the one that have a hyperelliptic component (except some sporadic examples in genus 3 and 4), and in this case there is only one supplementary component. The three types of strata containing hyperelliptic components are recalled on \S~\ref{sect:volhyp}.

The general strategy for computing Siegel--Veech constants for the connected strata can be adapted for connected components. For a connected component $ \cQ^{comp}(\alpha) $ we define the Siegel--Veech constants by the means:
\[c_*(\cQ^{comp}(\alpha), \cC)=\lim\limits_{\varepsilon\to 0}\frac{1}{2\pi\varepsilon^2}\frac{1}{\Vol(\cQ_1^{comp}(\alpha))}\int_{\cQ_1^{comp}(\alpha)}N(S, \cC, \varepsilon)\dd\mu_1(S).\] 
Note that the connected components of $ \cQ_1(\alpha) $ are exactly the intersection of $ \cQ_1(\alpha) $ with the connected components of $ \cQ(\alpha) $.
We have also the property that
\[c_{*}(\cC)=\lim\limits_{L\to\infty}\frac{N_*(S, \cC, L)\cdot(\mbox{Area of }S)}{\pi L^2},\]
for almost every $ S $ in the component $ \cQ^{comp}(\alpha) $.

So we will obtain the same evaluation: 
\begin{equation}c_*(\cQ^{comp}(\alpha),\cC)=\lim\limits_{\varepsilon\to 0}\frac{1}{2\pi\varepsilon^2}\frac{\Vol_*\cQ_1^{\varepsilon}(comp,\cC)}{\Vol\cQ_1(\alpha)}.\end{equation}\label{eq:SVbasecomp}

We apply this method in the case of hyperelliptic components in section \ref{sect:volhyp}. 

\section{Computation of Siegel-Veech constant for connected strata}\label{sect:SV}

In this section, $ \cQ(\alpha) $ will denote a connected stratum of quadratic differentials. We will evaluate Siegel--Veech constants $ c_*(\cC) $ defined in \S~\ref{ssect:genstrat} using equation (\ref{eq:SVbase}).

\subsection{Choice of normalization}\label{ssect:norm}
We have to choose a normalization for the volume element on a stratum $\cQ(\alpha)$, which is equivalent to choose a lattice in the space $H^1_-(\hat S, \hat\Sigma; \C)$ which gives the local model of the stratum $\cQ(\alpha)$ around $S$.

\begin{conv}\label{convreseau} We follow the convention of \cite{AEZ} and choose, as lattice in $H^1_-(\hat S, \hat\Sigma; \C)$ of covolume $1$, the subset of those linear forms which take values in $\Z \oplus i\Z$ on $H^-_1(\hat S, \hat\Sigma;\Z)$, that we will denote by $(H^-_1(\hat S, \hat\Sigma;\Z))^{*}_{\C}$.
\end{conv} 

This convention implies that the non zero cycles in $H_1(S, \Sigma,\Z)$ (that is, those represented by saddle connections joining two distinct singularities or closed loops non homologous to zero) have half-integer holonomy, and the other ones (closed loops homologous to zero) have integer holonomy. 

\begin{conv}\label{conv:label}We choose to label all zeros and poles. This affects the computation of volumes, but it is easy to deduce the value of volumes of strata with anonymous singularities.
\end{conv}

\subsection{Construction of a basis of $H_1^-(\hat S, \hat\Sigma,\Z)$}

In this section we recall the generic construction given in \cite{AEZ} of a basis of $H_1^-(\hat S, \hat\Sigma,\Z)$ from a basis of $H_1(S, \Sigma,\Z)$, and also a specific construction for each configuration. In the following sections we will look at every configuration and use the specific basis associated to each configuration in order to have a nice expression of the measure in terms of parameters of the cylinders. 

For a primitive cycle $[\gamma]$ in $H_1(S, \Sigma,\Z)$, that is, a saddle connection joining distinct zeros or a closed cycle (absolute cycle), the lift $[\hat\gamma]$ is a primitive element of $H_1^-(\hat S, \hat\Sigma,\Z)$.

\subsubsection{``Generic'' basis (cf \cite{AEZ} \S~3.1.)}
Let $ k $ be the number of poles in $ \Sigma $, $ a $ the number of even zeroes and $ b $ the number of odd zeros (of order $ \geq 1 $). Assume that the zeros are numbered in the following way: $P_1, \dots P_a$ are the even zeros, $P_{a+1}, \dots, P_{a+b}$ are the odd zeros and $P_{a+b+1}, \dots, P_n$ the poles, and take a simple oriented broken line $P_1, \dots P_{n-1}$. Take each saddle connection $\gamma_i$ represented by $[P_i, P_{i+1}]$ for $i$ going from $1$ to $n-2$, and a basis $\{\gamma_{n-1}, \dots, \gamma_{n+2g-2}\}$ of $H_1(S, \Z)$.

 \begin{lemma}\label{lem:basis}The family $\{\hat \gamma_1,\dots ,\hat \gamma_{n+2g-2}\}$ is a basis of $H_1^-(\hat S, \hat\Sigma,\Z)$.\end{lemma}

\begin{proof}First it is clear that the elements $ \hat \gamma_1,\dots ,\hat \gamma_{n+2g-2} $ are primitive elements of $H_1^-(\hat S, \hat\Sigma,\Z)$ and linearly independent. Moreover they do not generate a proper sub-lattice of $H_1^-(\hat S, \hat\Sigma,\Z)$.

 Each of the $ k $ poles lifts to a regular point in $ \hat S $ so does not appear in the list $ \hat\Sigma $. An even zero of order $ \alpha_i $ lifts to two zeros of degrees $ \frac{\alpha_i}{2} $, and an odd zero of order $ \alpha_j $ lifts to a zero of degree $ \alpha_j+1 $. So we have $ n=\vert\Sigma\vert=k+a+b $ and $ N=\vert\hat\Sigma\vert=2a+b $. Thus if $ \hat g $ is the genus of $ \hat S $ we have $ 4g-4=-k+\sum_{\alpha_i\geq 1}\alpha_i $ and $ 2\hat g-2=\sum_{\alpha_i\geq 1}\alpha_i +b $ and so \begin{eqnarray*} \dim_\C(H_1(\hat S, \hat\Sigma, \Z))=2\hat g -1+N &=&(2g-2+n)+(2g-1+a+b)\\ &=&\dim_\C H_1^-(\hat S,\hat\Sigma, \C)+\dim_\C H_1^+(\hat S, \hat\Sigma, \C).  \end{eqnarray*}
 
This equality on dimensions shows that we can complete the family $\{\hat \gamma_1, \dots , \hat \gamma_{n+2g-2}\}$ with $\{\gamma_1', \dots, \gamma_{n-k-1}', \gamma_{n-1}', \dots, \gamma_{n+2g-2}'\}$ to form a basis of $H_1(\hat S, \hat\Sigma, \R)$ (the linear independence is clear from the construction). The intersection matrix has integer coefficients and is of determinant 1, so that ends the proof of the lemma. 
\end{proof}

\subsubsection{Basis associated to a configuration}\label{goodbasis}
Fix a configuration $\cC$. As in \cite{EMZ}, we define an appropriate family $\{\gamma_1, \dots, \gamma_{n+2g-2}\}$ of $H_1(S, \Sigma,\Z)$ for $S\in \cC$, which lifts to a basis of $H_1^-(\hat S, \hat\Sigma,\Z)$, as follows: 
\begin{itemize}
\item for each component of the principal boundary stratum $\cQ(\alpha_i')$ take a family $\{\beta^i_1, \dots, \beta^i_{n_i+2g_i-2}\}$ of $H_1(S_i', \Sigma_i,\Z)$ such that $\{\hat \beta^i_1, \dots, \hat \beta^i_{n_i+2g_i-2}\}$ is a basis of $H_1^-(\hat S_i',\hat\Sigma_i,\Z)$ as previously,
\item for each \^homologous cylinder take a curve $\delta_j$ joining its boundary singularities (there might be an ambiguity in the choice of such a curve, cf \S~\ref{sssect:ambiguity})
\item take a saddle connection or a closed curve in the homology class of $\gamma$ (we denote $\pm v$ the holonomy of $\gamma$ ). 
\end{itemize} 
Lifting this basis to $H_1^-(\hat S, \hat\Sigma,\Z)$ using the  ~$\hat{ }$ operator provides a primitive basis of $H_1^-(\hat S, \hat\Sigma,\Z)$, as previously.

We will keep the same notations for elements in $(H^-_1(\hat S, \hat\Sigma;\Z))^{*}_{\C}$ 

\subsection{Computation}\label{ssection:computation}
Fix a configuration $\cC$ containing $q$ cylinders ($q\geq 1$). Now we give a complete description of the measure $\mu$ in terms of parameters of the configuration by disintegrating the volume element $\dd\mu$. 

By \cite{EM,MS} we have $\Vol_* \cQ_1^\epsilon (\cC)=\Vol_*\cQ_1^{\epsilon, thick }(\cC)+o(\epsilon^2)$, so we will describe $\mu$ only on $\cQ^{\epsilon, thick }(\cC)$.

Let $S\in\cQ^{\epsilon, thick}(\cC)$. Local coordinates near $S$ are given by $H^1_-(\hat S, \hat\Sigma,\C)$, and $\mu$ is just Lebesgue measure in this coordinates. 
 Choose now a basis associated to the configuration $\cC$ as above. It follows from the papers \cite{EMZ,MZ} that the measure $\dd\mu$ in $\cQ^{\epsilon, thick }(\cC)$ disintegrates as the product of the measure $\dd\mu'$ on $\cQ(\alpha')$ and a natural measure $\dd\nu_T$ on the space of parameters $\cT$ of the cylinders, that we describe in \S~3.3.1: $$\dd\mu=M'\dd\mu'\dd\nu_T$$ where $M'$ denotes the number of ways to get a surface $S$ in $\cQ^{\epsilon, thick }(\cC)$ when the parameters of the configuration are fixed. 

\subsubsection{Description of the space $\cT$ of the cylinders}\label{sssect:ambiguity}
Generally in a configuration, a labeling of the zeros and a choice of a covering path of the graph of the configuration induce a labeling of the cylinders. Sometimes some symmetries occur that exchange the cylinders but stabilize the zeros, they are taken into account in \S\ref{sect:count}. In the following we assume that the cylinders are numbered.

Roughly $\cT$ is described by coordinates $\pm v, h_1, \dots, h_{q}, t_1, \dots t_q$ representing the width, the heights and the twists of the cylinders, defined such that $h_i+it_i$ is the holonomy of the curve $\delta_i$. The problem here is that there might be an ambiguity for the choice of this curve and so for the definition of the twist. In the following we assume that the cylinders are horizontal, that is $ \pm v $ represents the horizontal direction in the surface $S$.
First note that despite the fact that the surface has a non trivial holonomy, for a given configuration $\cC$ it is possible to choose an orientation for each cylinder, for example by choosing an oriented path covering the graph representing the configuration. So in each cylinder we have a notion of bottom, up, left and right. Recall that thin cylinders are the one with each of their boundaries formed by a single saddle connection of holonomy $\pm v$, and so there is only one singularity on each of their boundaries. For these cylinders we can define the twist and the height of the cylinder as usual: starting from the only one singularity on the bottom of the cylinder, draw a vertical segment going up and ending at a point $P$ on the upper boundary of the cylinder. The length of this segment defines the height of the cylinder. Starting from the point $P$ and following the boundary in the right horizontal direction, we meet the singularity on the upper boundary of the cylinder, which is at distance $t$ from $P$, and $t$ defines the twist of the cylinder ($0\leq t< \vert v\vert$).
The next picture shows a particular case where the twist is ambiguous for a thick cylinder.

\begin{figure}[h!]
\centering
\includegraphics[scale=1]{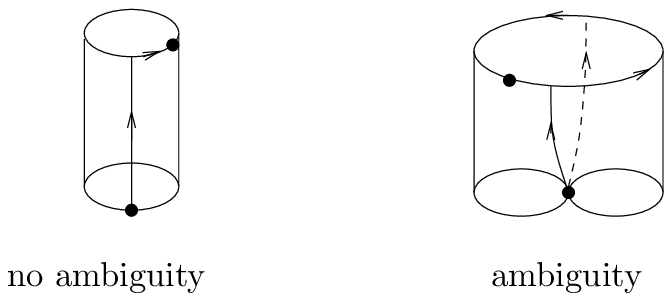}
\caption{Ambiguity for the definition of the twist}
\end{figure}

For the thick cylinders, we can define their twist as follows: for such a cylinder, if one of its boundaries contains two distinct singularities (recall that the singularities are labeled), then choose the one of the smaller index. We have now in each case one distinguished singularity on each of the two boundaries. Consider the shortest geodesic segments joining these two singularities (there might be two such segments). Then their vertical coordinates coincide and define the height $h$ of the cylinder, and their horizontal coordinate coincide modulo $\frac{2\vert  v \vert}{o_t}$, where $o_t=\vert\Gamma_{up}\vert\vee \vert\Gamma_{down}\vert$, and $\Gamma_{up}$ (resp. $\Gamma_{down}$) is the group of symmetries of the upper (resp. lower) boundary. In general for cylinders appearing in a configuration the orders of these groups are $1$ or $2$, so $o_t$ is equal to $1$ or $2$. In the example of the figure above, we have $ \vert\Gamma_{down}\vert=2 $, $ \vert\Gamma_{up}\vert=1 $ so $o_t=2$. So we define the twist as the value $t\in \left[ 0,  \frac{2\vert  v \vert}{o_t}\right)$ equal to the horizontal coordinates reduced modulo $\frac{2\vert  v \vert}{o_t}$. This ambiguity can appear only for thick cylinders having at least three or four boundary saddle connections, that is, cylinders of local type $ \circ 3.2 $ or $ \circ 4.2 $ in graphs of type $ c), d),$ or $ e) $ in the classification of Figure \ref{fig:fig3MZ} and Figure \ref{fig:fig6MZ}.

 We have $$\dd\nu_T=\dd hol(\hat{\gamma})\dd hol(\hat\delta_1) \dots \dd hol(\hat \delta_q).$$ 

Denote $n(q)$ the number of the cycles $\gamma, \delta_1, \dots, \delta_q$ in $H_1(S, \Sigma,\Z)$ that are not homologous to $0$ in $H_1(S, \Sigma,\Z)$. Taking care of the normalization (Convention \ref{convreseau}) we get: \begin{equation}\label{eq:dnut}\dd\nu_T=M_c\cdot\dd v\dd h_1\dots \dd h_q\dd t_1\dots \dd t_q\end{equation} with $M_c=4^{n(q)}$.

Note that with our choice of the basis, $\delta_1, \dots, \delta_q$ are always non homologous to zero. And $\gamma$ is homologous to zero if and only if  the associated graph of the configuration is of type $a$ in the classification of Masur and Zorich (Figure \ref{fig:fig3MZ}): in this case a vertex corresponding to a cylinder is separating the graph, and the boundary of any cylinder in the configuration consists of a single saddle connection (\^homologous to $\gamma$). So we have:
\begin{equation}M_c=\begin{cases}4^{q} \mbox{ if }\cC\mbox{ is of type }a\\ 4^{q+1} \mbox{ otherwise}\end{cases}\label{eq:M_c}\end{equation}

We choose to enumerate the cylinders such that the $q_1$ first cylinders have a waist curve of holonomy $\pm v$ and the $q_2$ remaining cylinders have a waist curve of holonomy $\pm 2 v$.

Consider now $\cT_1^\varepsilon$ the space of parameters of the cylinders with the additional constraint that the sum of the area of the \^homologous cylinders is normalized (i.e. equal to $1/2$) and that $|v|$ is bounded by $\epsilon$. Then the cone $C(\cT_1^\epsilon)$ underneath $\cT_1^\epsilon$ is given by the following equations:
\begin{eqnarray}\vert v\vert h\leq \frac{1}{2}\label{eq:area}\\
\vert v\vert \leq \varepsilon \sqrt{2\vert v\vert h}\label{eq:eps}\end{eqnarray}
where \[h=\sum_{k=1}^{q_1}h_k+2\sum_{k=1}^{q_2}h_{q_1+k}.\]

\subsubsection{Computation of $ c(\cC) $}\label{sssect:compc}
The volume of $ \cT_1^{\epsilon} $ is given by: $$\Vol(\cT_1^{\varepsilon})=\dim_{\R}(T)\nu_{T}(C(\cT_1^{\varepsilon}))=2(q+1)\nu_{T}(C(\cT_1^{\varepsilon}))$$
with \[\nu_{T}(C(\cT_1^{\varepsilon}))=\int_{C(\cT_1^{\epsilon})}\dd\nu_T\]
and $ \dd\nu_T $ given by (\ref{eq:dnut}).
Note that the measure $ \dd v $ on $ D_{\epsilon}/\pm $ disintegrates into $ w \cdot\dd w \cdot\dd\theta $ on $ [0,\epsilon]\times[0,\pi] $, and that integrating the measure of the twists $ \dd t_1\dots \dd t_q $ on $ [0,w)^{q_1}\times\displaystyle\prod_{i=q_1+1}^{q}\left[0,\cfrac{2w}{o_{t_i}}\right)  $ gives a factor $ \cfrac{2^{q_2}}{M_t}w^{q} $, with 
\begin{equation}M_t=\prod_{i=q_1+1}^{q}o_{t_i},\label{eq:M_t}\end{equation}
so we get:
\begin{equation*}\nu_{T}(C(\cT_1^{\varepsilon}))=M_c\pi\frac{2^{q_2}}{M_t}\int_{0}^{\frac{\epsilon}{2}}
w^{q+1}\dd w\int_{\R_{+}^{q}}  \chi{\left\{\cfrac{w}{2\varepsilon^2}\leq h\leq \cfrac{1}{2w}\right\}}  \dd h_1\dots \dd h_q.\end{equation*} 
With the following changes of variables $h_{q_1+k}'=2h_{q_1+k}$ we obtain:
\begin{equation*}\nu_{T}(C(\cT_1^{\varepsilon}))=\frac{M_c}{M_t}\pi\int_{0}^{\frac{\epsilon}{2}}
w^{q+1}\dd w\int_{\R_{+}^{q}}  \chi{\left\{\cfrac{w}{2\varepsilon^2}\leq h'\leq \cfrac{1}{2w}\right\}}  \dd h_1\dots \dd h_q'.\end{equation*} 
with \[ h'=\sum_{i=1}^{q_1}h_i+\sum_{i=1}^{q_2}h'_{q_1+i}.\]

Using the fact that \[\int_{\R_{+}^q}\chi\left\{a\leq \sum_{i=1}^{q} h_i \leq b\right\}\dd h_1\dots \dd h_q=\frac{1}{q!}(b^q-a^q),\]
since it is the difference of the volumes under two simplices in $ \R^q $, we obtain after computation:
\[\nu_{T}(C(\cT_1^{\varepsilon}))=\frac{M_c\pi\epsilon^2}{M_t2^{q+1}}\frac{q}{(q+1)!}\]

Thus: \[\Vol(\cT_1^{\varepsilon})=\frac{M_c\pi\varepsilon^2}{M_t2^{q}(q-1)!}.\]

We assume now that $\cQ(\alpha')$ is non empty, that is, the configuration $\cC$ is not made only by cylinders. Let $ S'\in\cQ_1(\alpha') $, then the rescaled surface $ r_S S' $ where $ 0< r_S\leq 1 $ has area $ \frac{r_S^2}{2} $. We define $ \Omega(\varepsilon, r_S) $ to be the subset of $ \cT $ formed by the cylinders rescaled such that gluing them to $ r_S S'$ after performing the appropriate surgeries gives a surface $ S\in C(\cQ_1^\varepsilon(\cC)) $. Note that the possible variations of area arising when performing the surgeries on $ r_S S' $ are negligible \cite{EMZ,MZ}. 

By definition $ \Omega(\varepsilon, r_S) $ is exactly formed by the rescaled surfaces $ r_T T $ where $ 0<r_T\leq 1 $, $ r_T^2+r_S^2\leq 1 $, and $ T\in\cT^{\tilde\varepsilon}_1 $, with $ \tilde\varepsilon=\varepsilon\sqrt{r_S^2+r_T^2} $. So we have, denoting $ Cusp(\varepsilon)=\Vol(\cT_1^{\varepsilon}) $,

\begin{eqnarray*}\nu_{T}(\Omega(\varepsilon,r_S))& = & \int_0^{\sqrt{1-r_S^2}}r_T^{2n_T-1}Cusp\left(\frac{\tilde\varepsilon}{r_T}\right)\dd r_T\\
& = & \frac{M_c\pi}{M_t2^{q}(q-1)!}\int_0^{\sqrt{1-r_S^2}}r_T^{2n_T-1}\varepsilon^2\frac{r_S^2+r_T^2}{r_T^2}\dd r_T\end{eqnarray*} with $n_T=\dim_{\C}(\cT)=q+1$, which simplifies:
\begin{equation}\nu_{T}(\Omega(\varepsilon,r_S))=\frac{M_c\pi\varepsilon^2}{M_t2^{q}(q-1)!}\int_0^{\sqrt{1-r_S^2}}r_T^{2q-1}(r_S^2+r_T^2)\dd r_T.\label{eq:nut}\end{equation}
After evaluation, we obtain: \[\nu_{T}(\Omega(\varepsilon,r_S))=\frac{M_c\pi\varepsilon^2}{M_t2^{q+1}(q+1)!}(1-r_S^2)^q(r_S^2+q).\]
Now if $ M_s $ denotes the number of ways to obtain a surface $ S\in C(\cQ_1^\varepsilon(\cC))$ by gluing $ r_T T \in\Omega(\varepsilon, r_S)$ to $ r_S S' \in\cQ(\alpha')$ (see (\ref{eq:M_s})), the total measure of the cone $ C(\cQ_1^\varepsilon(\cC)) $ is:
\begin{eqnarray}\mu(C(\cQ_1^\varepsilon(\cC)))&=&M_s\Vol(\cQ_1(\alpha'))\int_0^1 r_S^{2n_S-1}\nu_T(\Omega(\varepsilon,r_S))\dd r_S\label{eq:mu}\\ &=&\frac{M_sM_c\Vol(\cQ_1(\alpha'))\pi\varepsilon^2}{M_t2^{q+1}(q+1)!}\underbrace{\int_0^1 r_S^{2n_S-1}(r_S^2+q)(1-r_S^2)^q\dd r_S}_{I}\notag
\end{eqnarray}

An easy recurrence or a change of variables gives the following lemma:
\begin{lemma}\label{lem:J} $$J(a,q)=\int_0^1 r^{2a+1}(1-r^2)^q\dd r=\frac{1}{2}\frac{q!a!}{(a+q+1)!}$$
\end{lemma}
We recognize
$$I=J(n_S,q)+qJ(n_S-1,q).$$
After simplification we get:
$$I=\frac{(q+1)!(n_S-1)!}{2(n_S+q+1)!}(n_S+q).$$
So, denoting $M=\cfrac{M_sM_c}{M_t}$ we obtain:
$$\mu(C(\cQ_1^\varepsilon(\cC)))=M\pi\varepsilon^2\Vol(\cQ_1(\alpha'))\frac{(n_S-1)!(n_S+q)}{2^{q+2}(n_S+q+1)!}$$
As we have $$\Vol\cQ^\epsilon_1(\cC)=\dim_\R (\cQ(\alpha))\mu(C(\cQ_1^\epsilon(\cC)))$$ it follows from the definition of the Siegel--Veech constant that: $$c(\cC)=M\dim_\C(\cQ(\alpha))\frac{(n_S-1)!(n_S+q)}{2^{q+2}(n_S+q+1)!}\frac{\Vol \cQ_1(\alpha')}{\Vol \cQ_1(\alpha)}.$$
Recall that $\dim_{\C} \cQ(\alpha)=\dim_{\C} \cQ(\alpha')+\dim_{\C} \cT=n_S+q+1$. We obtain finally the formula (\ref{eq:cSV}) of Theorem \ref{th:ccyl}.


\subsubsection{Computation of $ c_{area}(\cC) $}\label{sssect:compcarea}
Here we want to compute $c_{area}(\cC)$, so we have to count surfaces with weight the area of cylinders with waist curve smaller than $ \varepsilon $, by definition. Note that, since there are $ q_1 $ cylinders of waist curve of length $ w=\vert  v\vert $ and $ q_2 $ of waist curve of length $ 2w $, if $ w\leq \frac{\varepsilon}{2} $ (when the area is renormalized), all cylinders count (with weight their area), and if $ \frac{\epsilon}{2}\leq w\leq \epsilon $, only the thin cylinders count (with weight their area). Equation (\ref{eq:eps}) contains two cases \begin{equation}\label{eq:domain1} w=\vert v\vert\leq \frac{\epsilon}{2}\sqrt{2\textrm{area}}\end{equation} and \begin{equation}\label{eq:domain2}\frac{\epsilon}{2}\sqrt{2\textrm{area}}\leq w\leq \epsilon\sqrt{2\textrm{area}}  \end{equation}  of different weights. So the domain of integration of $ C(T_1^{\epsilon}) $ splits into two parts as shown in the following picture. 

\begin{figure}[ht]
\centering
\includegraphics[scale=1]{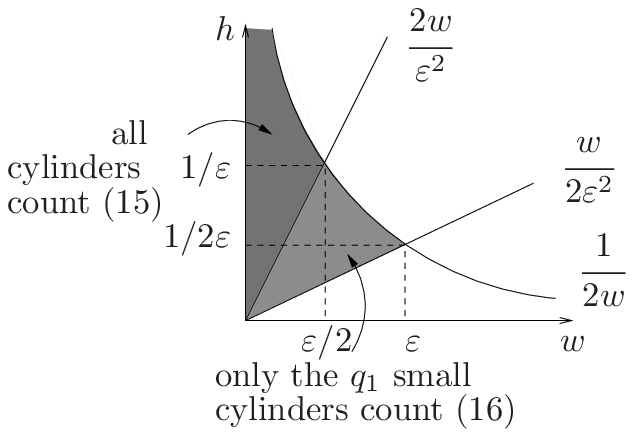}
\caption{Domain of integration}
\end{figure}
This gives the following weight function: \[W^{area}(w,h_i)=\begin{cases} \chi{\left\{\cfrac{2w}{\varepsilon^2}\leq h\leq \cfrac{1}{2w}\right\}}+\cfrac{\sum_{i=1}^{q_1}h_i}{h}  \chi{ \left\{ \cfrac{w}{2\varepsilon^2}\leq h\leq \cfrac{2w}{\epsilon^2} \right\} } \quad\textrm{if}\; w\leq\cfrac{\epsilon}{2},\notag\\
\cfrac{\sum_{i=1}^{q_1}h_i}{h} \chi{ \left\{ \cfrac{w}{2\varepsilon^2}\leq h\leq \cfrac{1}{2w} \right\} }  \quad\textrm{if}\; \cfrac{\epsilon}{2}\leq w\leq \epsilon .\notag \end{cases}\]

Now the weighted volume of $ \cT_1^{\epsilon} $ is given by: $$\Vol^{area}(\cT_1^{\varepsilon})=\dim_{\R}(T)\nu_{T}^{area}(C(\cT_1^{\varepsilon}))=2(q+1)\nu_{T}^{area}(C(\cT_1^{\varepsilon}))$$
with \[\nu_{T}^{area}(C(\cT_1^{\varepsilon}))=\int_{C(\cT_1^{\epsilon})}W^{area}(\vert v\vert, h_i)\dd\nu_T\]
and $ \dd\nu_T $ given by (\ref{eq:dnut}).

Following step by step the computations of the last paragraph, using the same change of variables, we have
\begin{multline*}\nu_{T}^{area}(C(\cT_1^{\varepsilon}))=\frac{M_c}{M_t}\pi\left[\int_{0}^{\frac{\epsilon}{2}}
w^{q+1}\dd w\int_{\R_{+}^{q}} \left(\chi{\left\{\cfrac{2w}{\varepsilon^2}\leq h'\leq \cfrac{1}{2w}\right\}} \right.\right.\\\left.  +\cfrac{\sum_{i=1}^{q_1}h_i}{h'}  \chi{ \left\{ \cfrac{w}{2\varepsilon^2}\leq h'\leq \cfrac{2w}{\epsilon^2} \right\} }\right)  \dd h_1\dots \dd h_q'  \\ \left. + \int_{\frac{\epsilon}{2}}^{\epsilon}w^{q+1}\dd w \int_{\R_{+}^{q}}\cfrac{\sum_{i=1}^{q_1}h_i}{h'} \chi{ \left\{ \cfrac{w}{2\varepsilon^2}\leq h'\leq \cfrac{1}{2w} \right\} } \dd h_1\dots \dd h_q'\right].\end{multline*} 

with \[ h'=\sum_{i=1}^{q_1}h_i+\sum_{i=1}^{q_2}h'_{q_1+i}.\]

Note that, since the variables $ h_i $ play symmetric roles, we have: 
\[\int_{\R_{+}^{q}}\cfrac{\sum_{i=1}^{q_1}h_i}{\sum_{i=1}^{q}h_i}\chi\left\{a\leq \sum_{i=1}^{q}h_i\leq b\right\}\dd h_1\dots\dd h_q=\cfrac{q_1}{q}\int_{\R_{+}^{q}}\chi\left\{a\leq \sum_{i=1}^{q}h_i\leq b\right\}\dd h_1\dots \dd h_q.\]
So computations are similar to the previous ones, and we obtain: 
\[\Vol^{area}(\cT_1^{\varepsilon})=\frac{M_c\pi\varepsilon^2}{M_t2^{q+2}q!}(4q_1+q_2).\]
Assume that $ \cQ(\alpha') $ is not empty. Now in (\ref{eq:nut}) we have to multiply the integrand by the ratio of the area of the cylinders by the total area of the surface $\frac{r_T^2}{r_S^2+r_T^2}$. We obtain: \begin{eqnarray}\nu^{area}_T(\Omega(\varepsilon,r_S)) &=\frac{M_c\pi\varepsilon^2(4q_1+q-2)}{M_t2^{q+2} q!}\int_0^{\sqrt{1-r_S^2}}r_T^{2q+1}\dd r_T\notag\\
&=\frac{M_c\pi\varepsilon^2(4q_1+q_2)}{M_t2^{q+2} q!}\frac{(1-r_S^2)^{q+1}}{2(q+1)}.\notag\end{eqnarray}
 Then: $$\mu^{area}(C(Q_1^\epsilon(\cC)))=M\Vol Q_1(\alpha')\frac{\pi\varepsilon^2(4q_1+q_2)}{2^{q+3}(q+1)!}\int_0^1(1-r_S^2)^{q+1}r_S^{2n_S-1}\dd r_S.$$
 Using again Lemma \ref{lem:J} we obtain: $$\mu^{area}(C(Q_1^\epsilon(\cC)))=M\Vol \cQ_1(\alpha')\frac{\pi\varepsilon^2(4q_1+q_2)}{2^{q+4}}\frac{(n_S-1)!}{(n_S+q+1)!}.$$
 So at the end we have: \begin{equation}c_{area}(\cC)=M\frac{4q_1+q_2}{2^{q+4}}\frac{(\dim_{\C} \cQ(\alpha')-1)!}{(\dim_{\C} \cQ(\alpha)-1)!}\frac{\Vol \cQ_1(\alpha')}{\Vol \cQ_1(\alpha)}.\label{careaC}\end{equation}
 
 Comparing to equation (\ref{eq:cSV}) and (\ref{eq:ccylc}) we obtain the relation (\ref{eq:careaandc}), which ends the proof of Theorem \ref{th:ccyl}.

\subsubsection{Special case}\label{sssection:specialcase} Assume that $\cQ(\alpha')$ is empty that is, the configuration is made only by cylinders. This arises only on strata $ \cQ(-1^ 4) $, $ \cQ(2, -1^ 2) $ and $ \cQ(2,2) $. Then the computations are much easier. Indeed we have in this case
  \[\Vol\cQ_1^\varepsilon(\cC)=\Vol \cT_1^\varepsilon=\frac{M_c\pi\varepsilon^2}{M_t2^{q}(q-1)!}\]
   and
    \[\Vol^{area} \cQ_1^\varepsilon(\cC)=\Vol^{area} \cT_1^\varepsilon=\frac{M_c\pi\varepsilon^2}{M_t2^{q+2}q!}(4q_1+q_2)\]
so, since the ratio of the area of the cylinders over the total area is $1$, we obtain the following proposition.
\begin{proposition} If the configuration $ \cC $ is only made by cylinders, then
 \begin{eqnarray} c(\cC)=\frac{M_c}{M_t2^{q+1}(q-1)!\Vol \cQ_1(\alpha)}\notag\\
 c_{cyl}(\cC)=\cfrac{4q_1+q_2}{4}c(\cC)\notag\\
 c_{area}(\cC)=\frac{1}{q}c_{cyl}(\cC)\label{eq:completecareacyl}\end{eqnarray}
  \end{proposition}
 
 \subsection{Volume of the boundary strata}

Consider a stratum $\cQ(\alpha)=\prod_{i=1}^m \cQ(\alpha_i) $ of disconnected flat surfaces. Following the notations of \cite{AEZ} and generalizing the result of 4.4 we obtain the following lemma:

\begin{lemma}
$$\Vol \cQ_1(\alpha)=\frac{1}{2^{m-1}}\frac{\prod (\dim_{\C}\cQ(\alpha_i)-1)!}{(\dim_{\C}\cQ(\alpha)-1)!}\prod_{i=1}^m\Vol \cQ_1(\alpha_i)$$
\end{lemma}

Let $  \cH_{r}(\alpha) $ be the hyperboloid of surfaces of area $ r $ in the Abelian stratum $ \cH(\alpha) $. We have the following relation between hyperboloids in this case:
\begin{lemma}\label{lem:H1/2}
$$\Vol \cH_{1/2}(\alpha)=2^{\dim_{\C} \cH(\alpha)}\Vol \cH_1(\alpha)$$
\end{lemma}

So the final formula for a boundary stratum $\cQ(\alpha')=\prod \cH(\alpha_i)\prod\cQ(\beta_j)$ with $m$ connected components is: 

\begin{equation}c_{area}(\cC)=M\frac{4q_1+q_2}{2^{m+q+3}}\frac{\prod_i(a_i-1)!2^ {a_i}\Vol \cH_1(\alpha_i)\prod_j(b_j-1)!\Vol \cQ_1(\beta_j)}{(\dim_{\C} \cQ(\alpha)-1)!\Vol \cQ_1(\alpha)}\label{eq:completecareaC}
\end{equation} where $a_i=\dim_{\C} \cH(\alpha_i)$ and $b_j=\dim_{\C} \cQ(\beta_j)$.
\subsection{Evaluation of $M_s$}

The general formula for $M_s$ is given by:
\begin{equation}M_s=\frac{K}{|\Gamma(\cC)|}\label{eq:M_s}\end{equation}
For each surface $S_i$ in the principal boundary, the number of geodesic rays coming from a boundary singularity on $S_i$ can be read on the local ribbon graph representing $S_i$: each boundary singularity is represented by a connected component of the local ribbon graph, summing the orders $k_{i_j}$ along this connected component gives the number of geodesic rays emerging form this singularity. If the surface as several boundary singularities, then one has to multiply the number of geodesic rays obtained for each of them, to get the combinatorial constant responsible for the gluing of $S_i$ in the configuration. Multiply the numbers obtained for each $S_i$ to get the final combinatorial constant $K$. Note that for surfaces of trivial holonomy the surgeries are made on rays pointing in the same direction, so there are less choices for the $ k_i $'s.

So for surfaces of non-trivial holonomy the constant $ K $ is given by the formula:
\[K=\prod_{bound. comp.} \sum_{k_i\; in\; b.c.} k_i, \]

and for surfaces of trivial holonomy the constant $ K $ is given by the formula:
\[K=2\prod_{bound. comp.} \sum_{k_i \; in\; b.c.} \cfrac{k_i}{2}. \]

$\Gamma(\cC)$ denotes the symmetries of the configuration $\cC$ that generalize the $ \gamma\mapsto -\gamma $ symmetry in the Abelian case. 
$$|\Gamma(\cC)|=\prod_i|\Gamma(S_i)|$$
and \[|\Gamma(S_i)|=\begin{cases}2 & \mbox{ if } S_i\mbox{ is a torus of trivial holonomy}\\2 & \mbox{ if } S_i\mbox{ is in a connected hyperelliptic stratum  and the surgery}\\ &\mbox{ applies to one or two fix points of the hyperellptic involution}\\
&\mbox{ or to two points exchanged by the involution}\\
1 & \mbox{ otherwise}\end{cases}\]

\subsection{Counting configurations}\label{sect:count}
Recall that by convention, all zeros and poles are numbered, so several configurations can share the same type $ \cC $ due to this labeling.

For each type of configuration $ \cC $, denote $ N(\cC) $ the number of configurations of this type: two configurations sharing the same type will be distinct if the label of one of the newborn singularities or  the subset of labels of interior singularities of one of the boundary surfaces differ in the two configurations.

For a connected stratum $ \cQ(\alpha) $ (or a connected component of a stratum), we have \[c_{area}(\cQ(\alpha))=\sum_{admissible\; \cC}N(\cC)c_{area}(\cC).\]
Recall that types of configurations identify with decorated global ribbon graphs embedded in the sphere described by Definition 3 of \cite{MZ}. Unless such a graph presents a {\it decorated ribbon graph symmetry}, there is a well-defined way to label the connected components of the ribbon graph and the boundary surfaces. 

We define a {\it decorated ribbon graph symmetry} as a symmetry of ribbon graph which preserves the decorations, that correspond here to the type of boundary surfaces ($ \oplus $ or $ \ominus $), the boundary singularities $ k_i $ and the set of interior singularities for each boundary surface. In the case of configurations these symmetries correspond to rotations of angle $ \pi $ of the sphere that the ribbon graphs are embedded in, so they are of order 2.
As an example, the following type of configuration possesses this symmetry.
\begin{figure}[h!]
\centering
\includegraphics[scale=1]{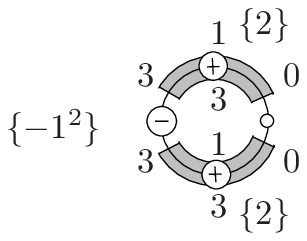}
\caption{Decorated ribbon graph symmetry}
\end{figure}
For types of configurations that do not posses this symmetry, $ N(\cC) $ is evaluated as follows. For a stratum $ \cQ(\alpha) $ with $ \alpha=\{d_1^{a_1}, \dots ,  d_m^{a_m}\}$, let $ \cC $ be a configuration without symmetry. Then we can label all boundary surfaces (say there are $ r $ such surfaces). For the boundary surface $ S_j $, let $ D^{int}_j=\{d_1^{a_1^j}, \dots, d_m^{a_m^j}\} $ denote the set of interior singularity orders of  $S_j$. We can also label the connected components of the ribbon graph (say there are $ s $ such components). Recall that each connected component of the ribbon graph corresponds to one or two newborn singularities (depending of the number of components of its boundary). For the $ k $-th connected component of the ribbon graph let $ D^{nb}_k=\{d_1^{b_1^k}, \dots d_m^{b_m^k}) $ denote the set of singularity orders of the newborn zeros (note that $ b_i^k \leq 2$). We have $ a_i=\sum_j a_i^j+\sum_k b_i^k $ and the number of ways to give names to the singularities is 
\begin{equation}N(\cC)=\prod_{i=1}^m \cfrac{a_i!}{\prod_{j=1}^r a_i^j !\cdot \prod_{k=1}^s b_i^k !}=\prod_{i=1}^m {a_i\choose a_i^1, \dots, a_i^r, b_i^1,\dots,b_i^s} .\label{eq:N}\end{equation}
For configurations that possess the symmetry, we have to divide this number by 2 if the symmetry acts non trivially on the connected components of the ribbon graph or on the boundary surfaces possessing interior singularities. If the symmetry stabilizes the connected components of the ribbon graph and the boundary surfaces, but acts non trivially on the cylinders, we also have to divide this number by 2 to take into account that there is no canonical numbering of the cylinders here (cf \S\ref{Q8} for an example).

On the previous example, the symmetry preserves the boundary surface $ \ominus $ but exchanges the two surfaces $ \oplus $ possessing an interior singularity of order $ 2 $. Here $ \alpha=\{-1^2, 2^2, 9^2\} $.
Thus for this type of configuration $ N(\cC) $ is given by
\[N(\cC)=\cfrac{1}{2}\cdot \cfrac{2!}{2!0!0!}\cdot\cfrac{2!}{0!1!1!}\cdot \cfrac{2!}{0!0!0!}=2.\]

\section{Strata $\cQ(1^k, -1^l)$, with $k-l=4g-4\geq 0$}\label{sect:princ}
The strata $\cQ(1^k, -1^l)$ are particularly interesting for two reasons. First, they correspond to strata of maximal dimension at genus and number of poles fixed. Second, their boundary strata belong to the same family, so that gives recursion formulas for Siegel--Veech constants and volumes. 

The strata $\cQ(1^2, -1^2)$ and $\cQ(1^4)$ are hyperelliptic and will be studied in \S~\ref{sect:volhyp}. In the general case there are only four types of configurations, so we give here their complete description and apply the formula for the Siegel--Veech constant $c_{area}(\cC)$ to each of them.
\subsection{Configurations}
\begin{proposition}There are only four types of configurations that contain cylinders for strata $\cQ(1^k, -1^l)$, they are described in Figure \ref{fig:confprincipal}.
\end{proposition}

\begin{figure}
\centering
\includegraphics[scale=1]{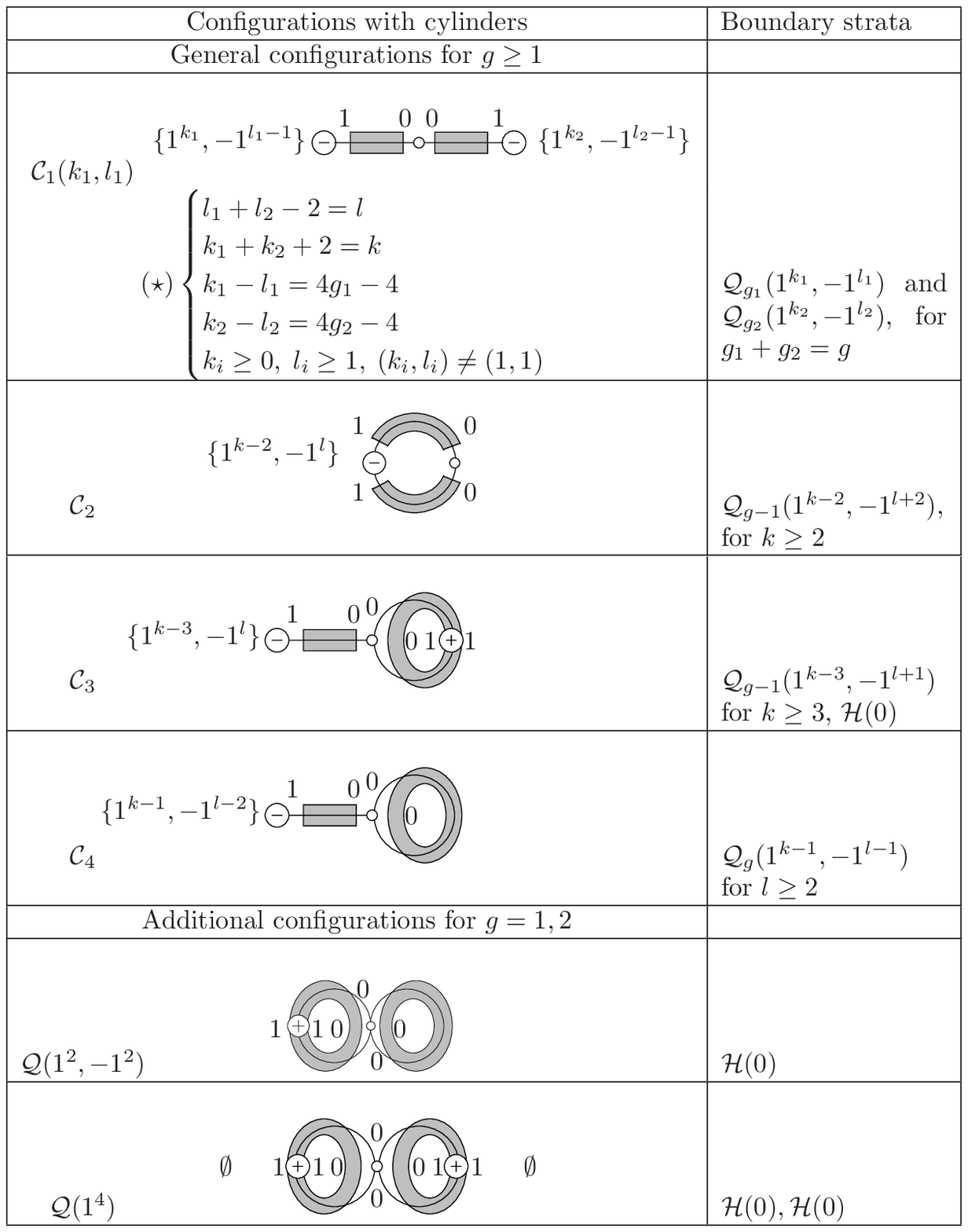}
\caption{Configurations containing cylinders for strata $\cQ(1^k, -1^l)$, with $k_l=4g-4$ and $g\geq 1$.}\label{fig:confprincipal}
\end{figure}

\begin{proof} We recall that graphs representing configurations are classified by Theorem 2 in \cite{MZ}. Then the proof is based on the observation that there not many ways to create zeros of order $1$ or poles (see also Lemma \ref{lem:odd:zeros} in \S~\ref{sect:geom}). We recall that the order of a newborn zero is given by the formula $\sum (k_i+1)-2$ where  the $k_i$ are the orders of the boundary singularities along the boundary component of the ribbon graph that corresponds to the newborn zero (see \S~1.4 of \cite{MZ} for more details), and we have $k_i\geq 0$. A boundary component admits at least one boundary singularity. 
So there is only one possibility for a pole: there is only one boundary singularity, which is equal to 0. For a zero of order 1 there are 3 possibilities: 
\begin{itemize}
\item one boundary singularity of order 2,
\item two boundary singularities of order 1 and 0,
\item three boundary singularities of order 0.
\end{itemize}
The first case is realizable when the global graph representing the configuration contains a loop with only one vertex. But in this case we can see that either there will be another newborn zero of higher order, or there will be no cylinders in the configuration.
The third case can also be eliminated because boundary components with exactly three boundary singularities arise only around vertices of type $+3.1$ in the graph, and the parities of the boundary singularities in this case are odd.

So the only remaining possibility is the second one. We can reformulate this discussion by saying that there is only one way to get a cone angle $3\pi$: one has to glue a cone angle $\pi$ with a cone angle $2\pi$. Looking carefully at all the ways to have boundary singularities of order 1 or 0 in the local ribbon graphs and the consequence on the boundary components in the global graph, we reduce the case to only two possibilities: the boundary singularity of order 0 arises only as a cone angle around points on the boundary of a cylinder, and the one of order 1 arises either by creating a hole adjacent to a pole in a surface of non trivial holonomy (i.e. for vertices of type $-1.1$ and $-2.2$), or by breaking up a marked point on a surface of trivial holonomy (i.e. for vertices of type $+2.1$). Note that the last surgery creates two points of cone angle $\pi$, so gluing each of them to a cylinder will create two newborn zeros. 

This situation is illustrated in the following pictures (Figure \ref{fig:creating1}).

\begin{figure}[h!]
\centering
\includegraphics[scale=1]{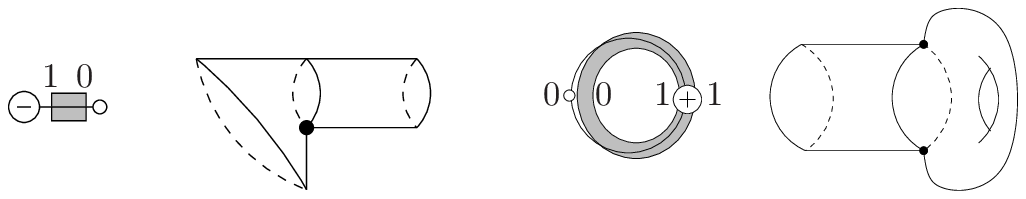}

\caption{Newborn zeros of order 1}\label{fig:creating1}
\end{figure}

For a pole, similar considerations give that there is only one way to get a pole (and not creating zeros of order $\ge 2$), by pinching the boundary of a cylinder (Figure \ref{fig:creatingpole}).

Note that, since the interior singularities are zeros of order $1$ or poles, the only boundary strata are $\cH(0)$ and $\cQ(1^{K}, -1^{L})$.

These remarks allow us to eliminate most of the configurations, and to keep only the four possible types of configurations described on Figure \ref{fig:confprincipal}.
\begin{figure}[h!]
\centering
\includegraphics[scale=1]{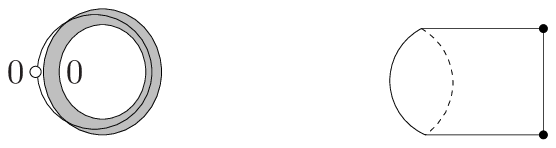} 
\caption{Newborn poles}\label{fig:creatingpole}
\end{figure}
\end{proof}

As an example, Table \ref{tab:boundstrata} details the boundary strata (except $\cH(0)$) of the stratum $\cQ(1^7, -1^3)$.

\begin{table}[h!]
\begin{tabular}{|r|r|ccccccc}
\cline{3-9}
\multicolumn{1}{r}{} &  & \multicolumn{7}{c}{Number of poles}\\
\cline{3-9}
\multicolumn{1}{r}{} &  & $0$ & 1 & 2 & 3 & 4 & 5 & 6 \\
 \hline
\multirow{3}{2.5mm}{\rotatebox{90}{Genus\;}} & $0$ & $\times$ & $\times$ & $\times$ & $\times$ & \cellcolor{Gray}\tiny{$\cQ(-1^4)$} & \cellcolor{Gray}\tiny{$\cQ(1, -1^5)$} & \tiny{$\cQ(1^2, -1^6)$} \\
& 1 & $\times$ & $\times$ & \cellcolor{Gray}\tiny{$\cQ(1^2, -1^2)$} & \cellcolor{Gray}\tiny{$\cQ(1^3, -1^3)$} & \cellcolor{Gray}\tiny{$\cQ(1^4, -1^4)$} & \cellcolor{Gray}\tiny{$\cQ(1^5, -1^5)$} & \tiny{$\cQ(1^6, -1^6)$} \\
  & 2 & \cellcolor{Gray}\tiny{$\cQ(1^4)$} & \cellcolor{Gray}\tiny{$\cQ(1^5, -1)$} & \cellcolor{Gray}\tiny{$\cQ(1^6, -1^2)$} & \cellcolor{GGray}\tiny{$\cQ(1^7, -1^3)$} & \tiny{$\cQ(1^8, -1^4)$} & \tiny{$\cQ(1^9, -1^5)$} & \tiny{$\cQ(1^{10}, -1^6)$} 
\end{tabular}
\bigskip 

\begin{tabular}{c}
\cellcolor{GGray}Stratum
\end{tabular}
\begin{tabular}{c}
\cellcolor{Gray}Boundary strata
\end{tabular}
\caption{Boundary strata of principal strata}\label{tab:boundstrata}
\end{table}

In general, the boundary strata of $\cQ(1^k, -1^l)$ are those of same genus with at most $l-1$ poles, those of lower genus with at most $l+2$ poles, and $\cH(0)$.

Note that, in this list, all values of volumes in genus 0 are known (cf \cite{AEZ}), equation (\ref{eq:ex1}) gives the values of volumes for the first entries in genus 1 and 2 (hyperelliptic case); and \cite{G} gives the values of the other strata of dimension up to 10 (cf \S~\ref{ss:exapp}).

\subsection{Siegel--Veech constants}

\begin{corollary}\label{th:SVprincipal}
Let $d=2g-2+k+l=\frac{1}{2}(3k+l)$ be the complex dimension of the stratum $\cQ(1^k, -1^l)$. The Siegel--Veech constants associated to the four configurations described in Figure \ref{fig:confprincipal} are the following:

\[c_{area}(\cC_1(k_1, l_1))=\frac{1}{4}\cfrac{(d_1-1)!(d_2-1)!}{(d-1)!}\cfrac{\Vol\cQ_1(1^{k_1}, -1^{l_1})\Vol\cQ_1(1^{k_2}, -1^{l_2})}{\Vol\cQ_1(1^k, -1^l)}\]
where $d_i=\dim_\C\cQ(1^{k_i}, -1^{l_i})=\frac{1}{2}(3k_i+l_i)$.

\[c_{area}(\cC_2)=2\cfrac{(d-3)!}{(d-1)!}\cfrac{\Vol\cQ_1(1^{k-2}, -1^{l+2})}{\Vol\cQ_1(1^k, -1^l)}\]

\[c_{area}(\cC_3)=\frac{\pi^2}{3} \cfrac{(d-5)!}{(d-1)!}\cfrac{\Vol\cQ_1(1^{k-3}, -1^{l+1})}{\Vol\cQ_1(1^k, -1^l)}\]

\[c_{area}(\cC_4) =  \frac{1}{2}\cfrac{(d-3)!}{(d-1)!}\cfrac{\Vol\cQ_1(1^{k-1}, -1^{l-1})}{\Vol\cQ_1(1^k, -1^l)}\]

If $ (k,l)\notin\{(2,2),(4,0)\} $, and if all the four configurations appear in a stratum $\cQ(1^k, -1^l)$, then the Siegel--Veech constant for the whole stratum is given by:

\begin{eqnarray*}c_{area}(\cQ(1^k, -1^l)) =  \sum_{\textrm{admissible }(k_1,l_1)}\cfrac{1}{2}\cdot\cfrac{k!\cdot l!}{k_1!k_2!(l_1-1)!(l_2-1)!}\cdot c_{area}(\cC_1(k_1, l_1))\\
  +  \cfrac{k(k-1)}{2}\cdot c_{area}(\cC_2)+\cfrac{k(k-1)(k-2)}{2}\cdot c_{area}(\cC_3)+\cfrac{kl(l-1)}{2}\cdot c_{area}(\cC_4)\end{eqnarray*}

\end{corollary}
For the additional configurations in genera 1 and 2, see \S~\ref{sect:volhyp}.

\begin{proof}
The proof is a straightforward application of Theorem \ref{th:ccyl} for configurations given in Figure \ref{fig:confprincipal}. In order to illustrate the theorem, we explain in details what are the combinatorial data and the possible symmetries for each configuration.
\begin{enumerate}
\item Configuration 1 (Figure \ref{fig:1k-1lconfig1}):

This configuration happens only for genus $ g\geq 1 $, and for $ k_i, l_i $ satisfying the constraints $ (\star) $. The last constraint excludes the stratum $ \cQ(1, -1) $ which is empty.  
Figure \ref{fig:1k-1lconfig1} shows on the left the ribbon graph encoding the configuration, and on the right a topological picture for this configuration. There are two boundary surfaces and two newborn singularities of order 1, corresponding to the two connected components of the ribbon, and produced by gluing a cylinder to a hole made on a pole of a boundary surface.

\begin{figure}[h!]
\centering
\includegraphics[scale=1]{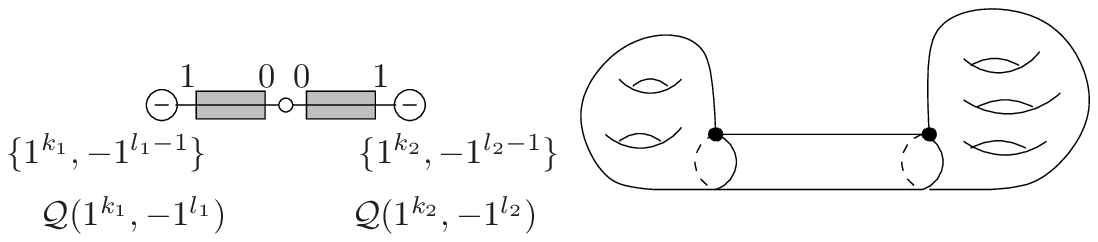}
\caption{\label{fig:1k-1lconfig1}Configurations $\cC_1(k_1,l_1)$ for $\cQ(1^k, -1^l)$ in genus $g \geq  1$}
\end{figure}

The configuration is of type $ a) $, the cylinder has its waist curve homologous to zero so by \eqref{eq:M_c}, we have $ M_c=4^{1} $. With this type of configuration there is no ambiguity for the twist so $ M_t=1 $ (see \eqref{eq:M_t}). There is only one choice for the ray we make the surgery along, and no local symmetry, so $ M_s =1$ (see \eqref{eq:M_s}).

We obtain the following combinatorial data for this configuration:
\begin{itemize}
\item  $M=\cfrac{M_c\cdot M_s}{M_t}=4$
\item $q_1=1, q_2=0$, since the cylinder is thin,
\item $\dim_{\C}\cQ(1^{k_i}, -1^{l_i})=2g_i-2+k_i+l_i=\frac{1}{2}(3k_i+l_i)=d_i$ 
\end{itemize}

Applying formula \eqref{eq:completecareaC}, with $ m=2 $ (two connected components for the boundary stratum) we get:
\[c_{area}(\cC_1(k_1, l_1))=4\cdot \frac{4}{2^{6}}\cfrac{(d_1-1)!(d_2-1)!\Vol\cQ_1(1^{k_1}, -1^{l_1})\Vol\cQ_1(1^{k_2}, -1^{l_2})}{(d-1)!\Vol\cQ_1(1^k, -1^l)}\]

If $ (k_1,j_1)\neq(k_2,l_2) $ the number of configurations of this type is (see \S~\ref{sect:count}) is obtained by applying formula \eqref{eq:N} with $ D_1^{nb}=D_2^{nb}=\{1\} $, $ D_1^{int}=\{1^{k_1}, -1^{-l_1-1}\} $ and $ D_2^{int}=\{1^{k_2}, -1^{-l_2-1}\} $:
\[N(\cC_1(k_1,l_1))=\cfrac{k!}{k_1!k_2!}\cdot\cfrac{l!}{(l_1-1)!(l_2-1)!}.\]

If $ (k_1,l_1)=(k_2,l_2)=(\frac{k}{2}-1,\frac{l}{2}+1) $, there is a decorated ribbon graph symmetry that exchanges the two boundary surfaces and the two connected components of the ribbon. In this case
\[N(\cC_1(\frac{k}{2}-1,\frac{l}{2}+1)=\cfrac{1}{2}\cdot \frac{k!}{\left(\left(\frac{k}{2}-1\right)!\right)^2}\cdot\frac{l!}{\left(\left(\frac{l}{2}\right)!\right)^2}.\]

Noting that the configuration $ \cC(k-k_1-2,l-l_1+2)$ is the symmetric of the configuration $ \cC(k_1,l_1) $, the contribution of these types of configurations for all admissible $ (k_1,l_1) $ is \[\sum_{\textrm{admissible }(k_1,l_1)}\cfrac{1}{2}\cdot\cfrac{k!}{k_1!k_2!}\cdot\cfrac{l!}{(l_1-1)!(l_2-1)!}\cdot c_{area}(\cC_1(k_1,l_1).\]

\item Configuration 2 (Figure \ref{fig:1k-1lconfig2}):

This configuration happens only for genus $ g\geq 1 $ and for number of zeros $ k\geq 2 $.
\begin{figure}[h!]
\centering
\includegraphics[scale=1]{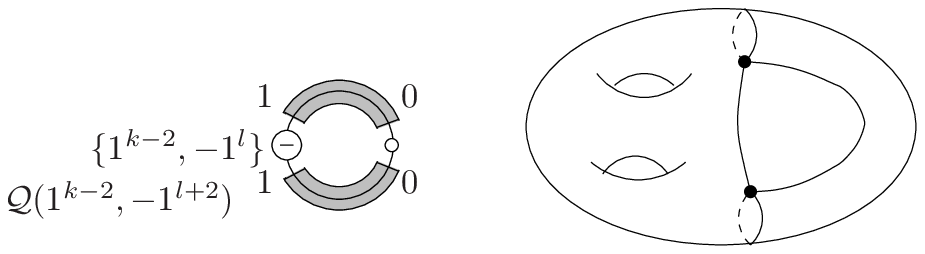}

\caption{\label{fig:1k-1lconfig2}Configuration $\cC_2$ for $\cQ(1^k, -1^l)$ in genus $g\geq 1$ and $k\geq 2$}
\end{figure}

Here the waist curve of the (thin) cylinder is non homologous to zero. We get the following combinatorial data:

\begin{itemize}
\item $M_c=4^2$, $M_t=1$, $M_s=1$
\item $q_1=1, q_2=0$
\item $\dim_{\C}\cQ(1^{k-2}, -1^{l+2})=2g+k+l-4=d_2$
\end{itemize}

We get:
\[c_{area}(\cC_2)=4^2\cdot\frac{4}{2^5}\cfrac{(d-3)!\Vol\cQ_1(1^{k-2}, -1^{l+2})}{(d-1)!\Vol\cQ_1(1^k, -1^l)}\]

There is here a decorated ribbon graph symmetry that stabilizes the boundary stratum and exchanges the two connected components of the ribbon, that is the two newborn singularities. Also we have
\[N(\cC_2)=\frac{1}{2}\cdot\frac{k!}{(k-2)!}\]
configurations of this type (see \S~\ref{sect:count}).

\item Configuration 3 (Figure \ref{fig:1k-1lconfig3}):

This configuration happens only for genus $ g\geq 1 $ and for number of zeros $ k\geq 3 $.
\begin{figure}[h!]
\centering
\includegraphics[scale=1]{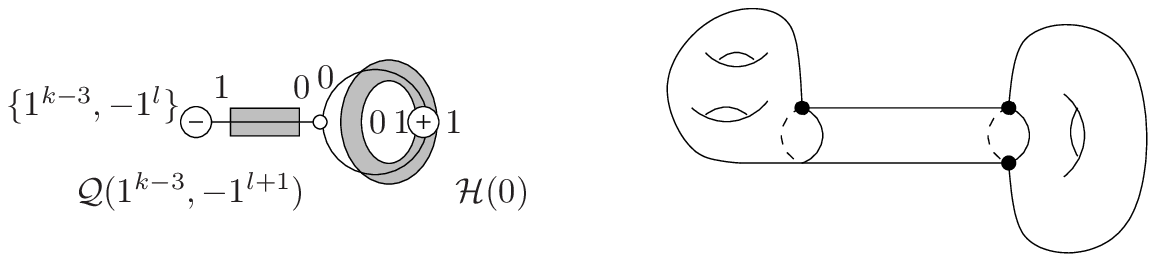}

\caption{\label{fig:1k-1lconfig3}Configuration $\cC_3$ for $\cQ(1^k, -1^l)$ in genus $g\geq 1$ and $k\geq 3$}
\end{figure}

Note that here the cylinder is thick but its twist is not ambiguous to define since the newborn zeros on the right are distinct, so $ M_t=1 $.

The shortest saddle connections defining the cylinder are the two joining the two newborn zeros on the right: they are not homologous to zero, whereas the saddle connection joining the newborn zero to itself on the left, which is \^homologous to the others, is homologous to zero. With our choice of convention we get $ M_c=4^{2} $ (see \eqref{eq:M_c}).

Note that the boundary stratum $ \cH(0) $ presents a local symmetry: the two possible rays to make the surgery are map one to another by the involution of the torus, so $ M_s=\frac{2}{2}=1 $.

Thus the combinatorial data are:
\begin{itemize}
\item $M=16$, 
\item $q_1=0, q_2=1$
\item $\dim_{\C}\cQ(1^{k-3}, -1^{l+1})=2g+k+l-6=d-4$
\item $\Vol\cH_{1/2}(0)=\frac{4\pi^2}{3}$ (see Lemma \ref{lem:H1/2}), and $ \dim_\C\cH(0)=2 $.
\end{itemize}

Applying formula \eqref{eq:completecareaC} with $ m=2 $ we get:
\[c_{area}(\cC_3) =  16\cdot\frac{1}{2^6}\cfrac{(d-5)!\Vol\cQ_1(1^{k-3}, -1^{l+1})(2-1)!\Vol\cH_{1/2}(0)}{(d_1)!\Vol\cQ_1(1^k, -1^l)}\]

The number of configurations of this type is:
\[N(\cC_3)=\frac{k!}{2!(k_3)!}.\]

\item Configuration 4 (Figure \ref{fig:1k-1lconfig4}):

\begin{figure}[h!]
\centering
\includegraphics[scale=1]{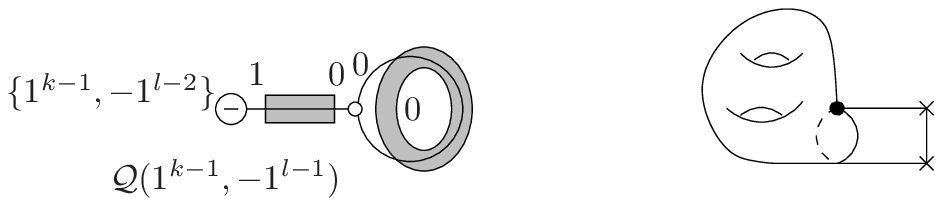}
\caption{\label{fig:1k-1lconfig4}Configuration $\cC_4$ for $\cQ(1^k, -1^l)$ in genus $g\geq 1$ and $l\geq 2$}
\end{figure}

The combinatorial data are:
\begin{itemize}
\item $M_c=4^2$, $M_t=1$, $M_s=1$
\item $q_1=0, q_2=1$
\item $\dim_{\C}\cQ(1^{k-1}, -1^{l-1})=2g+k+l-4=d-2$
\end{itemize}

Theorem \ref{th:ccyl} gives:
\[c_{area}(\cC_4) =  4^2\frac{1}{2^5}\cfrac{(d-3)!\Vol\cQ_1(1^{k-1}, -1^{l-1})}{(d-1)!\Vol\cQ_1(1^k, -1^l)}\]

There are
\[N(\cC_4)=\frac{k!}{(k-1)!}\cdot\frac{l!}{2!(l-2)!}\] configurations of this type.
\end{enumerate}
After simplification of the formulas we obtain the results of Corollary \ref{th:SVprincipal}.
\end{proof}
\subsection{Example of application}\label{ss:exapp}
As an application of the previous results, we compute the first steps of the recursion and obtain the exact Siegel--Veech constants for the first strata, using the values of the volumes computed in \cite{G}. The results match the approximate values obtained by experiments on Lyapunov exponents, provided by Charles Fougeron. Table \ref{tab:volSV} gathers all these data.

\begin{table}[h!]
\renewcommand{\arraystretch}{1.5}
\[\begin{array}{|c|c|c|}
\hline
\text{Stratum} & \text{Volume} & \pi^2\cdot c_{area}  \\
\hline
\cQ(1^3, -1^3) & 11/60\cdot  \pi^6 & \simeq 2.134\\
\hline
\cQ(1^4, -1^4) & 1/10\cdot  \pi^8 & \simeq 2.096\\
\hline
\cQ(1^5, -1) & 29/840\cdot  \pi^8 & \simeq 2.642\\
\hline
\cQ(1^5, -1^5) & 163/3042 \cdot \pi^{10} & \simeq 2.122\\
\hline
\cQ(1^6, -1^2) & 337/18144\cdot  \pi^{10} & \simeq 2.413\\
\hline
\end{array}
\]
\caption{Table of known values of volumes \cite{G}, and approximate values of Siegel--Veech constants}\label{tab:volSV}
\end{table}

We start with the stratum $\cQ(1^3, -1^3)$. 
Corollary \ref{th:SVprincipal} gives:
\begin{eqnarray*}
c_{area}(\cQ(1^3, -1^3))  =  3c_{area}(\cC_2)+3c_{area}(\cC_3)+9c_{area}(\cC_4)\\
 =  \frac{3}{10}\frac{\Vol\cQ_1(1, -1^5)}{\Vol\cQ_1(1^3, -1^3)}+\frac{\pi^2}{120}\frac{\Vol\cQ_1(-1^4)}{\Vol\cQ_1(1^3, -1^3)}+\frac{9}{40}\frac{\Vol\cQ_1(1^2, -1^2)}{\Vol\cQ_1(1^3, -1^3)}
\end{eqnarray*}

Using values \begin{eqnarray*}\Vol\cQ_1(1^k,-1^{k+4})=\cfrac{\pi^{2k+2}}{2^{k-1}} & \mbox{ \cite{AEZ}}, \\ \Vol\cQ_1(1^2, -1^2)=\frac{\pi^4}{3} & \mbox{ (\ref{eq:ex1})},\end{eqnarray*} we get:

\[c_{area}(\cQ(1^3, -1^3))=\frac{47}{120}\cfrac{\pi^4}{\Vol\cQ_1(1^3, -1^3)}\]

Using the value of the volume given in Table \ref{tab:volSV}, we obtain
\[c_{area}(\cQ(1^3, -1^3))=\cfrac{47}{22\cdot \pi^2},\]
which matches the approximated value given in Table \ref{tab:volSV}.

Similarly for the other strata we obtain exact values of Siegel--Veech constants that match the approximated ones. Table \ref{tab:SVLyap} gives all these exact values, as well as the exact values of the sums of Lyapunov exponents for the Hodge bundle over the strata along the Teichm\"uller flow (using Theorem 2 of \cite{EKZ}). In this table we denote \[L^{+}=\lambda_1^{+} +\dots +\lambda_g^{+}\] the sum of the Lyapunov exponents $\lambda_1^{+}\geq\dots\geq \lambda_g^{+}$ of the invariant subbundle $H_1^{+}$ of the Hodge bundle with respect to the involution induced by the natural involution on the double cover surfaces, and \[L^{-}=\lambda_1^{-} +\dots +\lambda_{g_{\eff}}^{-}\] the Lyapunov exponents $1=\lambda_1^{-}\geq\dots \geq \lambda_{g_{\eff}}^{-}$ of the anti-invariant subbbundle $H_1^{-}$ (see \cite{EKZ} for the definitions).
For surfaces of genus 1, we get the exact value of $\lambda_1^{+}$ which is very useful for the study of windtree models (see \cite{Delecroix:Hubert:Lelievre}), on the other cases, we obtain bounds for individual Lyapunov exponents.

\begin{table}[h!]
\renewcommand{\arraystretch}{1.5}
\[\begin{array}{|c|c|c|c|c|c|}
\hline
\text{Stratum} & g&g_{\eff} &\pi^2\cdot c_{area} &L^+ & L^- \\
\hline
\cQ(1^3, -1^3) & 1 & 3&{47}/{22}& 6/11&17/11  \\
\hline
\cQ(1^4, -1^4) &1 &4 & {44}/{21}&10/21&38/21\\
\hline
\cQ(1^5, -1) & 2 &4 & {230}/{87 }&32/29&154/87\\
\hline
\cQ(1^5, -1^5) &1 &5&{2075}/{978 }&70/163&1025/489 \\
\hline
\cQ(1^6, -1^2) & 2 &5 & {8131}/{3770 }&1041/1885&2926/1885\\
\hline
\end{array}
\]
\caption{Table of obtained exact values of Siegel--Veech constants and sums of Lyapunov exponents}\label{tab:SVLyap}
\end{table}

\section{Formulas for hyperelliptic components}\label{sect:volhyp}

\subsection{Volumes of hyperelliptic components}
The strata of the moduli spaces of Abelian differentials have at most three components: in genus $ g\geq 4 $ there are three connected component when the stratum possesses an hyperelliptic component and a well-defined spin structure (i.e. the zeros are even), there are two components when the stratum  possesses either a hyperelliptic component or a well-defined spin structure, but not both, and one component in all remaining cases \cite{KZ}. In lower genus the strata $ \cH(1,1) $ and $ \cH(2) $ are hyperelliptic and connected, and the strata $ \cH(2,2) $ and $ \cH(4) $ have two connected components.

We recall from \cite{G} the formulas for the volumes of hyperelliptic components in the Abelian case.

\begin{proposition}The volumes of hyperelliptic components of strata of Abelian differentials with area $1/2$ are given by the following formulas:
\begin{eqnarray}\Vol^{numb}\cH^{hyp}_{1/2}(k-1)=\cfrac{2^{k+2}}{(k+2)!}\cdot\cfrac{(k-2)!!}{(k-1)!!}\cdot\pi^{k+1}\label{eq:volHhyp1}\\
\Vol^{numb}\cH^{hyp}_{1/2}\left(\left(\frac{k}{2}-1\right)^2\right)=\cfrac{2^{k+3}}{(k+2)!}\cdot\cfrac{(k-2)!!}{(k-1)!!}\cdot\pi^{k}\label{eq:volHhyp2}\end{eqnarray}
\end{proposition}

\begin{remark}
Note that the hyperelliptic involution is a natural symmetry for the surfaces in $\cH^{hyp}(k-1)$, if we choose to count them modulo this symmetry, as in \cite{EO}, the volume of this component should be divided by 2. For the second type, labelling the zeroes kills this symmetry.
\end{remark}

The strata of the moduli spaces of quadratic differentials have one or two connected components: for genus $g\geq 5$ there are two components when the stratum contains a hyperelliptic component \cite{L2}. For genus $g\leq 4$ some strata are hyperelliptic and connected \cite{L1}: namely $\cQ(1^2, -1^2)$ and $\cQ(2, -1^2)$ in genus 1, $\cQ(1^4)$, $\cQ(2, 1^2)$, and $\cQ(2,2)$ in genus 2.

We recall from \cite{G} the formulas for the volumes of hyperelliptic components in the quadratic case.
\begin{proposition} 
The volumes of hyperelliptic components of strata of quadratic differentials are given by the following formulas:
\begin{itemize}
\item First type ($k_1\geq -1$ odd, $k_2\geq -1$ odd, $(k_1, k_2)\neq (-1, -1)$):

If $k_1\neq k_2$:
\begin{equation}\Vol^{numb}\cQ_1^{hyp}(k_1^2, k_2^2)=\frac{2^{d}}{(d)!}\pi^{d}\frac{k_1!!}{(k_1+1)!!}\frac{k_2!!}{(k_2+1)!!}\label{eq:volQhyp1}\end{equation}
Otherwise:
\begin{equation}\Vol^{numb}\cQ_1^{hyp}((g-1)^4)=\frac{3\cdot 2^{2g+2}}{(2g+2)!}\pi^{2g+2}\left(\frac{(g-1)!!}{g!!}\right)^2\label{eq:volQhyp1part}\end{equation}

\item Second type ($k_1\geq -1$ odd, $k_2\geq 0$ even):
\begin{equation}\Vol^{numb}\cQ_1^{hyp}(k_1^2, 2k_2+2)=\frac{2^{d}}{(d)!}\pi^{d-1}\frac{k_1!!}{(k_1+1)!!}\frac{k_2!!}{(k_2+1)!!}\label{eq:volQhyp2}\end{equation}

\item Third type ($k_1$, $k_2$ even):

\begin{equation}\Vol^{numb}\cQ_1^{hyp}(2k_1+2, 2k_2+2)=\frac{2^{d+1}}{(d)!}\pi^{d-2}\frac{k_1!!}{(k_1+1)!!}\frac{k_2!!}{(k_2+1)!!}\label{eq:volQhyp3}
\end{equation}
\end{itemize} where $d=k_1+k_2+4$ is the complex dimension of the strata.
\end{proposition}

\begin{example}
For the five strata that are connected and hyperelliptic we obtain:
\begin{eqnarray}\Vol\cQ_1(1^2, -1^2)=\cfrac{\pi^4}{3}=30\zeta(4) & \Vol\cQ_1(1^4)=\cfrac{\pi^6}{15}=63\zeta(6) \label{eq:ex1}\\
\Vol\cQ_1(2, -1^2)=\cfrac{4\pi^2}{3}=8\zeta(2) & \Vol\cQ_1(2, 1^2)=\cfrac{2\pi^4}{15}=12\zeta(4) \label{eq:ex2}\\
& \Vol \cQ_1(2, 2)=\cfrac{4\pi^2}{3}=8\zeta(2) \label{eq:ex3}
\end{eqnarray}

\end{example}

\begin{remark}
Note that if surfaces are counted modulo symmetries, as in \cite{EO2}, then the volume of the third type of hyperlliptic component should be divided by 2.
\end{remark}


\subsection{Configurations containing cylinders in hyperelliptic components}
The complete list of all configurations of \^homologous saddle connections is described by C.~Boissy in \cite{B}. We extract from this list the configurations containing cylinders, and recall them on Figure \ref{fig:confighyp}.

The following proposition precises the boundary of the hyperelliptic components of strata.

\begin{proposition}\label{prop:boundhyp}Let $S$ be a flat surface in a hyperelliptic component of a stratum of quadratic differentials $\cQ^{hyp}(\alpha)$. Les $\gamma$ be a collection of \^homologous saddle connections realizing a configuration $\cC$ on the previous list (Figure \ref{fig:confighyp}). Then the two possible boundary components $S_1, S_2 \in \cQ(\alpha'_1), \cQ(\alpha'_2)$ of $S$ are hyperelliptic.

For every surfaces $S_1\in\cQ^{hyp}(\alpha'_1)$, $S_2\in\cQ^{hyp}(\alpha'_2)$, there is at least one way to assemble $S_1$ and eventually $S_2$ following configuration $\cC$ to obtain a hyperelliptic surface $S$.
\end{proposition}

\begin{proof} If $S\in\cQ^{hyp}(\alpha)$, following Lemma 10.3 of \cite{EMZ}, we may assume that the hyperelliptic involution fixes each boundary component. So it implies that $S_1$ and $S_2$ are also hyperelliptic.

If $S_1\in\cQ^{hyp}(\alpha'_1)$ and $S_2\in\cQ^{hyp}(\alpha'_2)$, we can make the surgeries on the boundary surfaces in such a way that the new surfaces stay invariant under the hyperelliptic involution (cf \S~14 in \cite{EMZ}). Then we construct an application on $S$ that acts on each boundary component as the hyperelliptic involution for the corresponding stratum and on the cylinder either by fixing its boundaries and rotating or by exchanging its two boundaries depending on the configuration $\cC$, in such a way that the global application is an involution of $S$. The action of the hyperelliptic involution on the configurations is detailed in \cite{B}.

\end{proof}

Note that the complex dimension of any hyperelliptic component is given by: 
\begin{align*}d:=\dim_\C\cQ^{hyp}(k_1^ 2, 2k_2+2)=\dim_\C\cQ^{hyp}(k_1^2, k_2^2)\\= \dim_\C\cQ^{hyp}(2k_1+2, 2k_2+2)=k_1+k_2+4.\end{align*}

First recall that the constants for the entire components are known \cite{EKZ}:
\begin{lemma}
\begin{equation}c_{area}(\cQ^{hyp}(\alpha))=\frac{k_1+k_2+4}{4\pi^2}\left(2+\frac{1}{(k_1+2)(k_2+2)}\right)\end{equation}\label{eq:chyp}
for $\alpha=(k_1^2, k_2^2)$, $\alpha=(k_1^2, 2k_2+2)$ or $\alpha=(2k_1+2, 2k_2+2)$.
\end{lemma}

\begin{proof}It is a direct consequence of Corollary 3 in \cite{EKZ}.
Let $L^-$ denote the sum of the Lyapunov exponents $\lambda_1^-, \dots, \lambda_{g_{\mathrm{eff}}}^-$ for the hyperelliptic component $\cQ^{hyp}(\alpha)$. Recall that by Theorem 1 of \cite{EKZ}, we have: \[c_{area}(\cQ^{hyp}(\alpha))=\frac{3}{\pi^2}(L^--I-K)\] where \[I=\frac{1}{4}\sum_{d_j\;odd}\frac{1}{d_j+2},\quad K=\frac{1}{24}\sum_j \frac{d_j(d_j+4)}{d_j+2}.\]
Corollary 3 in \cite{EKZ} gives the values of $L^-$ for hyperelliptic components, that we recall here:
\begin{eqnarray*}L^-=\frac{k_1+k_2+4}{4}\left(1+\frac{1}{(k_1+2)(k_2+2)}\right) & \mathrm{  for} & \cQ^{hyp}(k_1^2, k_2^2)\\
L^-=\frac{k_1+k_2+4}{4}\left(1+\frac{1}{k_1+2}\right) & \mathrm{  for} & \cQ^{hyp}(k_1^2, 2k_2+2)\\
L^-=\frac{k_1+k_2+4}{4} & \mathrm{  for} & \cQ^{hyp}(2k_1+2, 2k_2+2)\end{eqnarray*}
\end{proof}

For hyperelliptic components we obtain the following variation of formula (\ref{eq:completecareaC}):

\begin{proposition}
Let $ \cC $ be an admissible configuration for a hyperelliptic component of a stratum $ \cQ(\alpha) $ (see Figure \ref{fig:confighyp} and Figure \ref{fig:confighypdeg}). If $ \cQ(\alpha)\neq \cQ(2, -1^2) $ and $ \cQ(\alpha)\neq \cQ(2, 2) $, then the corresponding Siegel--Veech constant is given by:
\begin{equation}c_{area}(\cC)=M\frac{4q_1+q_2}{2^{m+q+3}}\frac{\prod_i(a_i-1)!2^{a_i}\Vol \cH_1^{hyp}(\alpha_i)\prod_j(b_j-1)!\Vol \cQ_1^ {hyp}(\beta_j)}{(\dim_{\C} \cQ(\alpha)-1)!\Vol \cQ_1^{hyp}(\alpha)}\label{eq:completecareaChyp}\end{equation}

where $M=\frac{M_sM_c}{M_t}$ and $M_c$, $M_t$ are given by (\ref{eq:M_c}) and (\ref{eq:M_t}), and $ M_s $ is given by Figure \ref{fig:confighyp} and Figure \ref{fig:confighypdeg}.

For the connected strata $ \cQ(2, -1^2) $ and $ \cQ(2,2) $ the configurations and the corresponding Siegel--Veech constants are given on Figure \ref{fig:confighypdeg}.

\end{proposition}

Formula \eqref{eq:completecareaChyp} is applied to each configuration on Figures \ref{fig:confighyp} and \ref{fig:confighypdeg}: it is easy to see that in each case the sum on all admissible configurations gives the known constant \eqref{eq:chyp} for the entire component.

\begin{figure}
\centering
\includegraphics[scale=1]{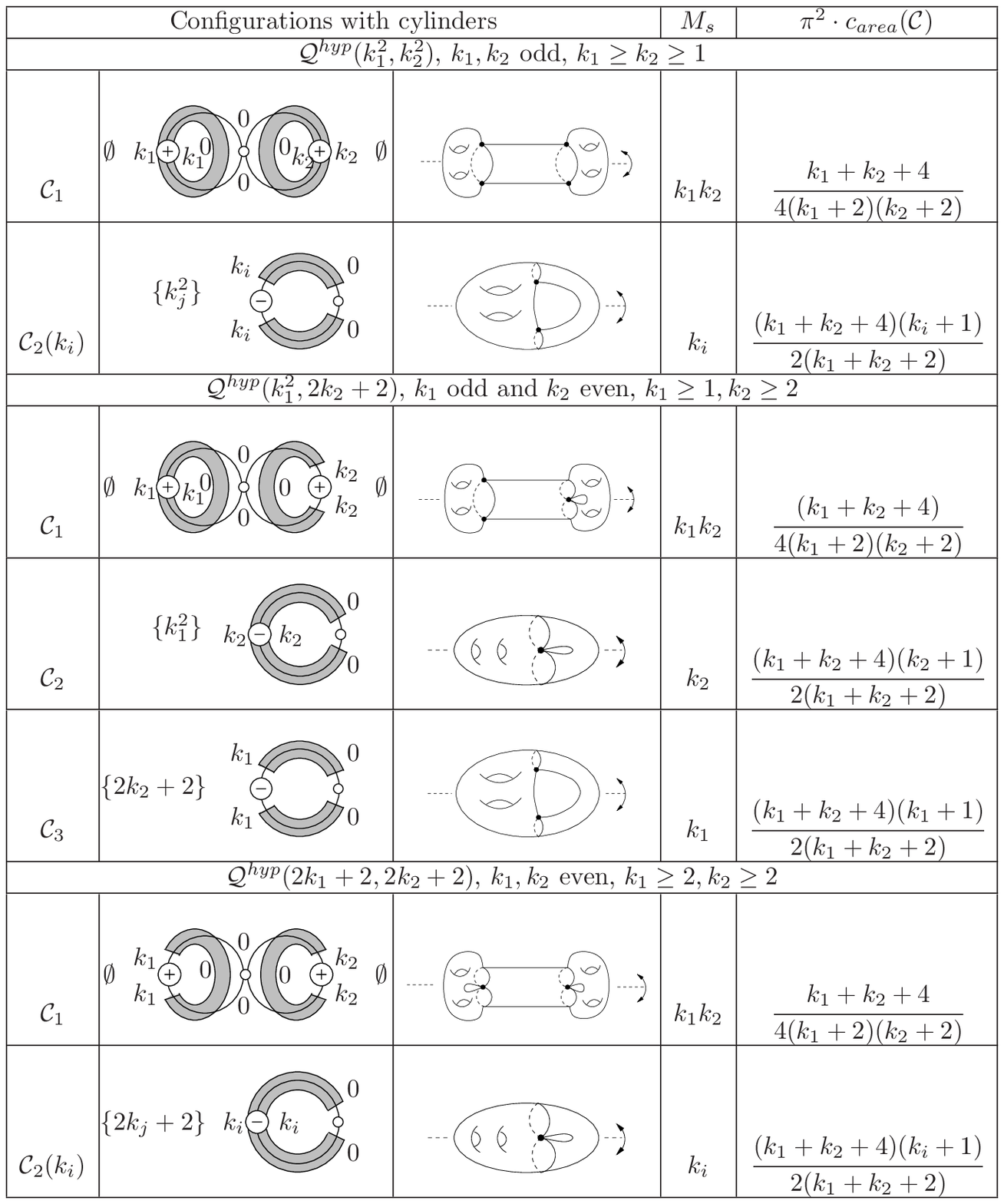}
\caption{Configurations containing cylinders for hyperelliptic components of strata of quadratic differentials. For all these configurations, $N(\cC)=1$.}\label{fig:confighyp}
\end{figure}

\begin{figure}
\centering
\includegraphics[scale=1]{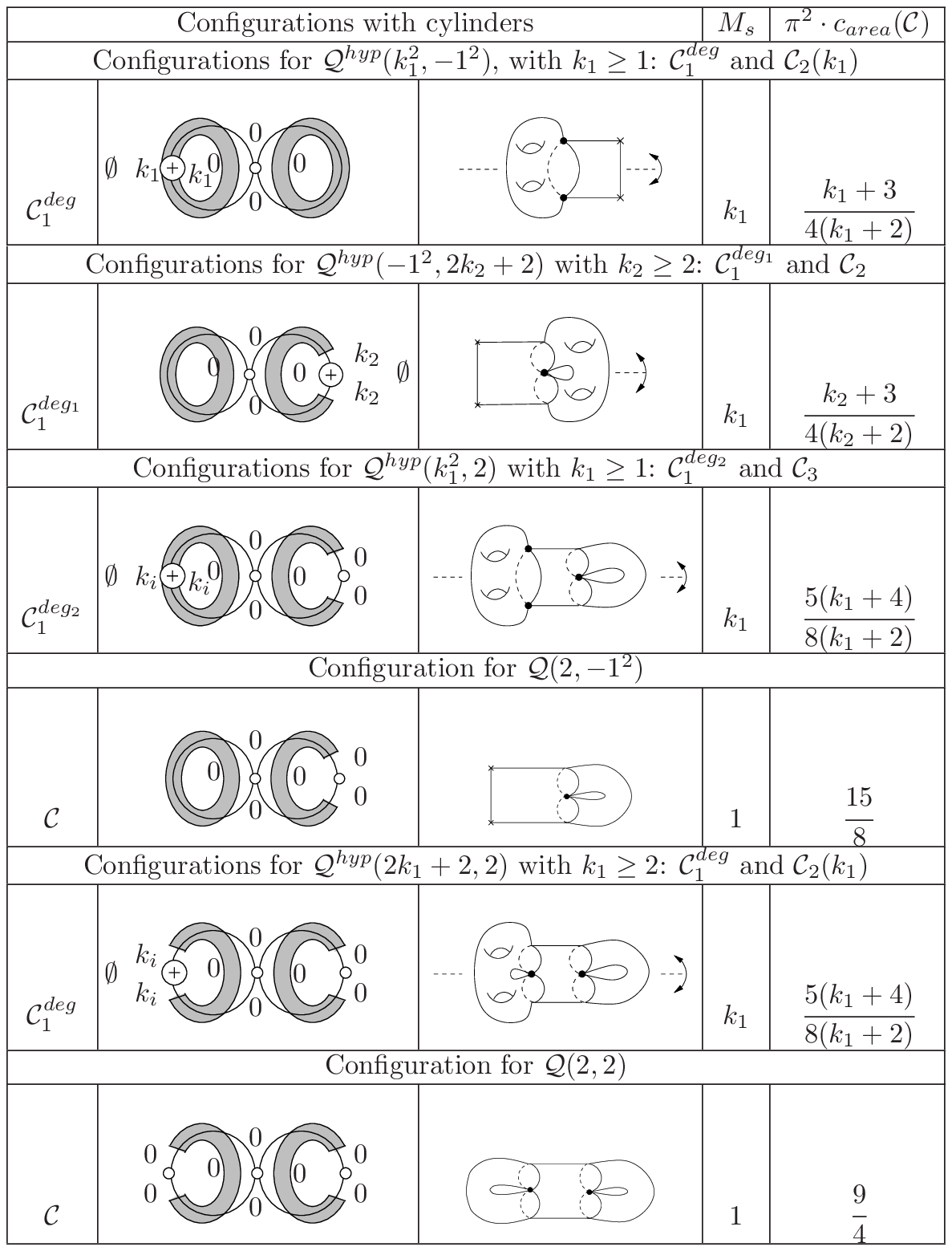}
\caption{Configurations containing cylinders for hyperelliptic components of strata of quadratic differentials in some particular cases. For all these configurations, $N(\cC)=1$.}\label{fig:confighypdeg}
\end{figure}

\begin{proof}
To compute the constants for each configuration, we use the method described in \S~\ref{ssect:noncon}, we follow step by step the computations of \S~\ref{sect:SV} and make only a few adjustments.

First assume that the boundary stratum of $ \cQ^{hyp}(\alpha) $ is not empty.
 Then this boundary is described by Proposition \ref{prop:boundhyp} and consists of hyperelliptic components of the boundary strata of $ \cQ(\alpha) $, so $ \Vol_*\cQ_1^{\varepsilon}(comp,\cC) $ is expressed in terms of $ \prod_i \Vol\cQ^{hyp}(\alpha_i') $. We have to take care of the symmetries induced by the hyperelliptic involution, which only change the constant $ M_s $ giving the number of ways to glue surfaces to cylinders to obtain a configuration.
 
 Consider the configuration $ \cC_1 $ for the component $ \cQ^{hyp}(k_1^2, k_2^2) $. The hyperelliptic involution stabilizes each boundary component $\cH^{hyp}(k_i-1) $ and acts on it as the hyperelliptic involution of the component. For each boundary surface there are $ 2k_i $ (non-oriented) horizontal rays emerging from the singularity, so only $ k_i $ choices for the surgery, since the hyperelliptic involution induces a symmetry of order 2. So for this configuration $ M_s=k_1k_2 $.
 For the configuration $ \cC_2(k_i) $ the hyperelliptic involution exchanges the two newborn singularities, so the two holes in the boundary surface. To perform the two holes surgery on the boundary surface, once we have chosen one of the $ k_i $ horizontal rays emerging from a singularity, we have to take for the other singularity the geodesic ray which corresponds to the the first under the action of the hyperelliptic involution on the boundary surface. So $ M_s=k_i $ (instead of $ k_i^2 $ for the configuration in the general case).
 For the other cases, the result is similar to these two first cases, so we do not repeat the arguments.
 
 If the boundary stratum is empty, that is, the configuration is made only by cylinders, which happens only for the connected strata $ \cQ(2, -1^2) $ and $ \cQ(2,2) $, we apply formula \eqref{eq:completecareacyl} of \S~\ref{sssection:specialcase}.
\end{proof}

\section{Examples of application for strata of small dimension}\label{sect:ex}

In this section we illustrate Theorem \ref{th:ccyl} for strata of small dimension.

\subsection{Volumes}\label{volexample}
Table \ref{tab:vol} gathers data on Siegel--Veech constants and volumes for strata of dimension 4 to 6: the lower dimension strata are $\cQ(-1^4)$ which corresponds to genus 0 and $\cQ(2, -1^2)$ which is hyperelliptic.

 The exact values of Siegel--Veech constants are given for the non-varying strata and the hyperelliptic components of strata. In \cite{CM} Chen and M\"oller define a stratum to be non varying when the sum of the Lyapunov exponents for any Teichm\"uller curve in the stratum is equal to the sum of Lyapunov exponents for the entire stratum. For strata possessing this property they give the value of the sum of exponents, so the value of the Siegel--Veech constant is obtained by applying the result of Eskin--Kontsevich--Zorich \cite{EKZ}.

The approximated values of Siegel--Veech constants are computed using experimental values for the sum of the exponents provided by Anton Zorich.

The exact values of the volumes are extracted from \cite{G}, the approximated ones come from \cite{DGZZ}.

\begin{table}[h!]
\renewcommand{\arraystretch}{1.5}
\[\begin{array}{|c|c|c||c|c|c|}
\hline
\text{Stratum} & \text{Volume} & \pi^2\cdot c_{area} &  \text{Stratum} & \text{Volume}&\pi^2\cdot c_{area} \\
\hline
\hline
\multicolumn{3}{|c||}{\text{Dimension }4} &\multicolumn{3}{c|}{\text{Dimension }6}\\
\hline

 \multicolumn{3}{|c||}{\text{genus }1}& \multicolumn{3}{c|}{\text{genus }1}\\
\hline
\cQ(1^2, -1^2) & \pi^4/3 & 7/3&  \cQ(1^3, -1^3) & 11\pi^6/60&\simeq 2.137\\
\hline
\cQ(3, -1^3) & 5\pi^4/9& 9/5 & \cQ(3, 1, -1^4)&\pi^6/3&59/30\\
\hline
\multicolumn{3}{|c||}{\text{genus }2} &\cQ(2^2, -1^4)& 136\pi^6/45 &\simeq 1.985\\
\hline
\cQ(2^2) & 4\pi^2/3 &9/4&\cQ(5, -1^5)&7\pi^6/10 &27/14\\
\hline
\cQ(5,-1) &28\pi^4/135 &15/7& \multicolumn{3}{c|}{\text{genus }2} \\
\hline
\multicolumn{3}{|c||}{}&\cQ(1^4)&\pi^6/15&19/6\\
\hline

 \multicolumn{3}{|c||}{\text{Dimension }5}& \cQ(3, 1^2, -1)&\pi^6/9& 79/30 \\
\hline
\multicolumn{3}{|c||}{\text{genus }1}& \cQ(2^2, 1, -1) &4\pi^4/5& 29/12\\
\hline
\cQ(2, 1, -1^3) & \pi^4 & 49/24 & \cQ(5,1,-1^2)&7\pi^/30& 97/42\\
\hline
\cQ(4, -1^4)& 2 \pi^4& 11/6 & \cQ(4, 2, -1^2) &28\pi^4/15& 53/24\\
\hline
 \multicolumn{3}{|c||}{\text{genus }2} &\cQ^{hyp}(3^2, -1^2) &\pi^6/30&33/10 \\
\hline
\cQ(2, 1^2)& 2\pi^4/15& 65/24 &\cQ^{non}(3^2, -1^2)&22\pi^6/135&21/10\\
\hline
\cQ(4, 1,-1) & 8\pi^4/15 & 5/2&\cQ(7, -1^3)&27\pi^6/50& 37/18\\ 
\hline
\cQ(3,2,-1)& 10 \pi^4/27 & 87/40 &\multicolumn{3}{c|}{\text{genus }3}\\
\hline
\cQ^{hyp}(6, -1^2)& 8\pi^4/45 & 45/16  &\cQ(7,1)&18\pi^6/175& 49/18\\
\hline
\cQ^{non}(6, -1^2)&32\pi^4/27 & 33/16 &\cQ^{hyp}(6, 2) &16\pi^4/135& 51/16\\
\hline
\multicolumn{3}{|c||}{\text{genus }3}&\cQ^{non}(6, 2) &96\pi^4/135& 39/16 \\
\hline
\cQ(8) & 10 \pi^4/27 &12/5 & \cQ(5,3) &14\pi^6/243& 171/70\\
\hline
\multicolumn{3}{c|}{}&\cQ(4^2) & 4\pi^4/5 &8/3\\
\cline{4-6}
\multicolumn{3}{c|}{}&\cQ^{reg}(9,-1)&\simeq 0.297\pi^6& 51/22\\
\cline{4-6}
\multicolumn{3}{c|}{}&\cQ^{irr}(9,-1)&\simeq 0.064\pi^6 & 63/22\\
\cline{4-6}
\end{array}
\]
\caption{Table of volumes of strata \cite{G} and Siegel--Veech constants \cite{CM}}\label{tab:vol}
\end{table}

We illustrate the main result of this paper on these small dimensional strata. In the case of non-varying strata, we find the known values for the entire strata, in the other cases, we obtain new exact values of Siegel--Veech constants. Note that this procedure can be reversed to obtain volumes from Siegel--Veech constants.

\subsection{Dimension 4}
The only two strata of dimension 4 of genus at least 1 that are not hyperelliptic are $ \cQ(3, -1^3) $ and $ \cQ(5,-1) $, and they are non-varying. For both of these strata we detail all configurations with cylinders and give the corresponding Siegel--Veech constants. We use the values of volumes given above to check the coherence of the formulas for these examples.

\subsubsection{$\cQ(3, -1^3)$}

The configurations are detailed on Figure \ref{fig:tab3iii}. For each of these configurations we give the combinatorial constant $ M $, the corresponding Siegel--Veech constant given by Theorem \ref{th:ccyl}, and the number of configurations for each type.

\begin{figure}[h!]
\centering
\includegraphics[scale=1]{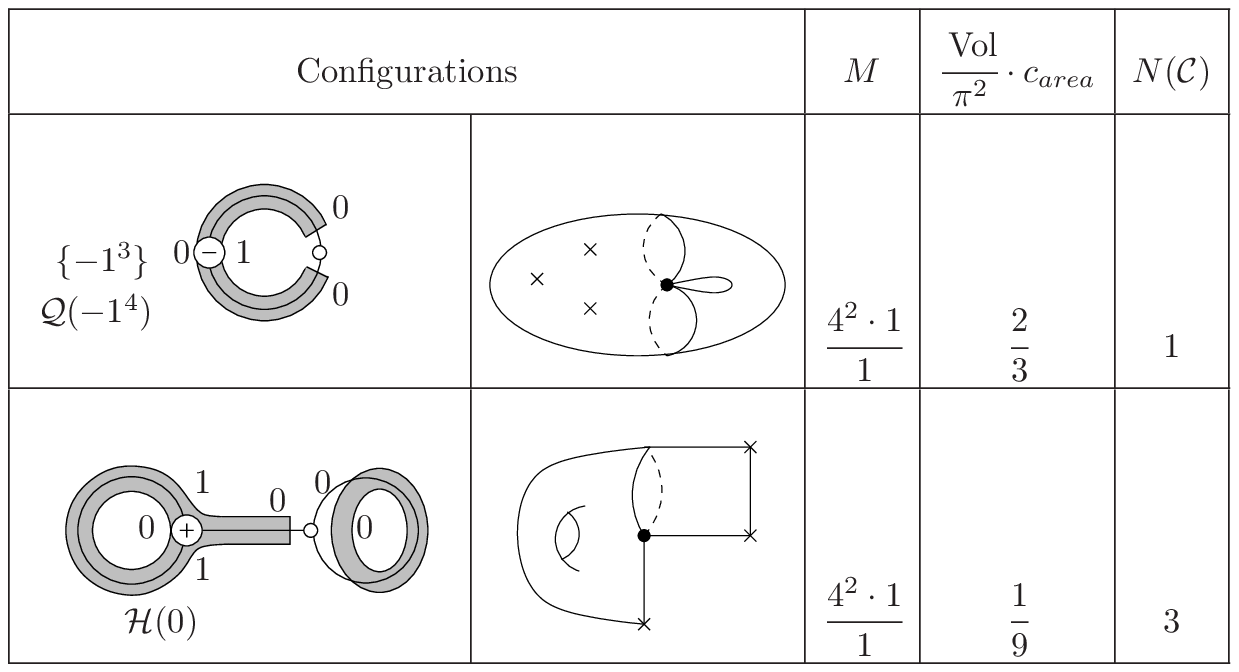}

\caption{Configurations containing cylinders for $\cQ(3, -1^3)$ and associated Siegel--Veech constants}\label{fig:tab3iii}
\end{figure}

Summing on all configurations we obtain:
\[c_{area}(\cQ(3, -1^3))=\cfrac{\pi^2}{\Vol\cQ_1(3, -1^3)}.\]

Using the value of the volume $ \Vol\cQ_1(3, -1^3)=\cfrac{5\pi^4}{9} $ (\S~\ref{volexample}) we get the known value of the Siegel--Veech constant for the stratum.

\subsubsection{$\cQ(5, -1)$}
The only one configuration is given in Figure \ref{fig:tab5i}.
\begin{figure}[h!]
\centering
\includegraphics[scale=1]{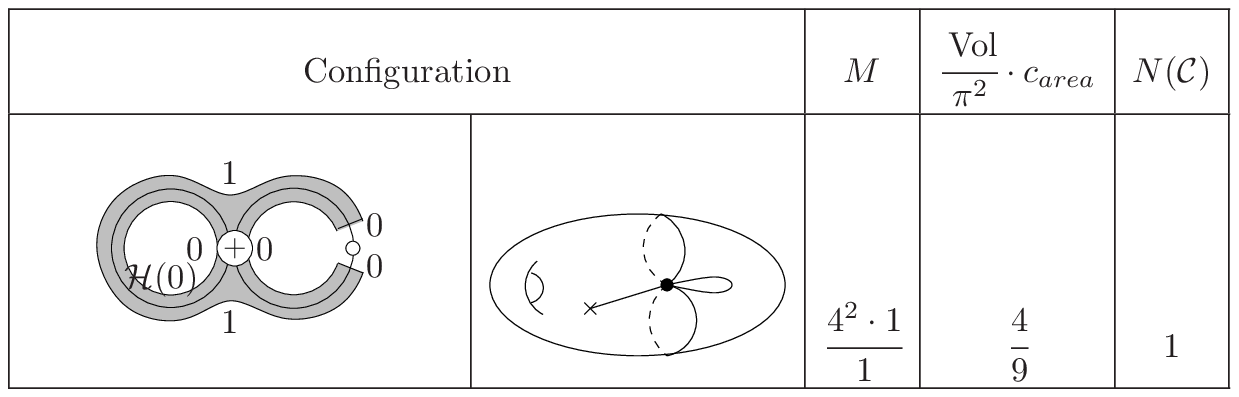}

\caption{Configuration containing cylinders for $\cQ(5, -1)$ and associated Siegel--Veech constant}\label{fig:tab5i}
\end{figure}

Using the value of the volume $ \Vol\cQ_1(5, -1)=\cfrac{28\pi^4}{135} $ given in \S~\ref{volexample}, we obtain the known value of the Siegel--Veech constant for the entire stratum.

\subsection{Dimension 5}
There are seven strata of dimension 5 and genus at least 1. We detail here the configurations for all these strata, except for the stratum $ \cQ(2, 1^2) $ which is hyperelliptic and connected. They are all non-varying.    

\subsubsection{$\cQ(2, 1, -1^3)$}
All configurations with cylinders for this stratum are given on Figure \ref{fig:tab21iii}.
\begin{figure}[h!]
\centering
\includegraphics[scale=1]{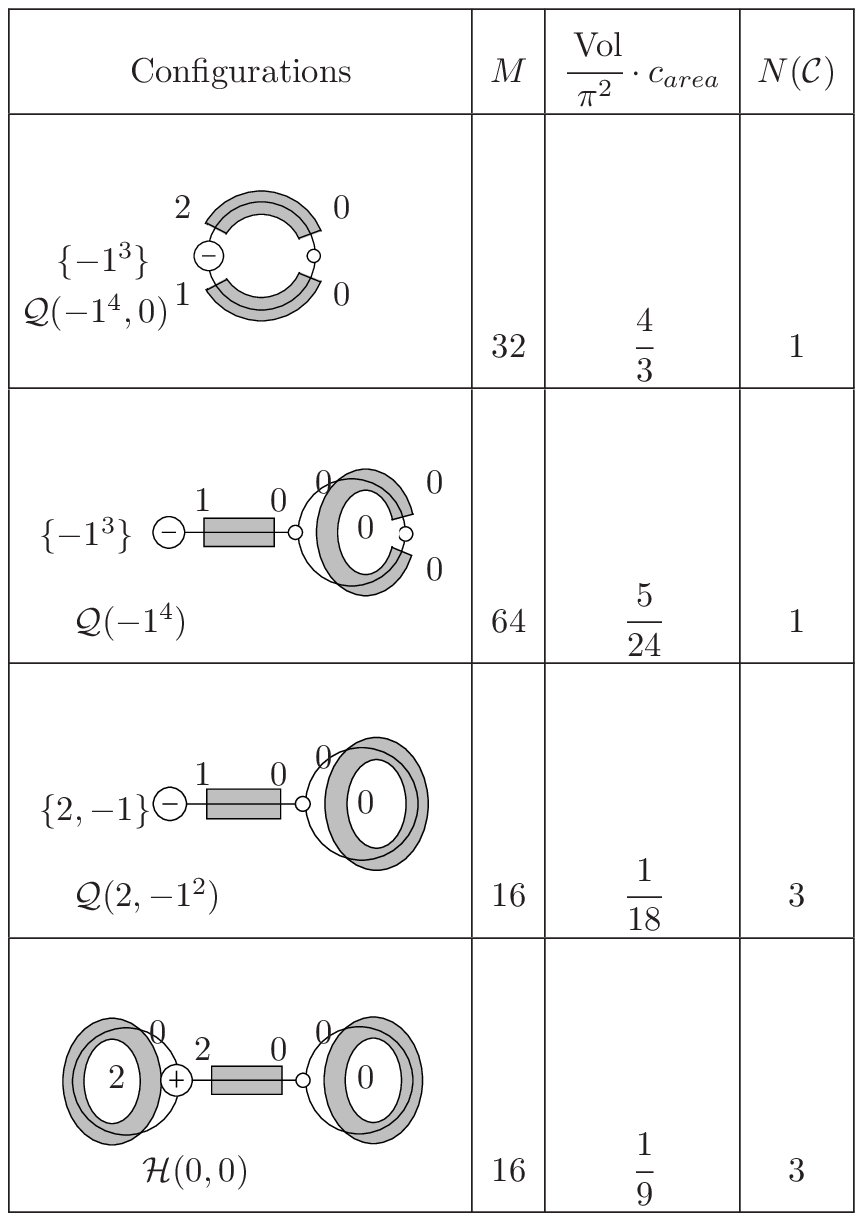}

\caption{Configurations containing cylinders for $\cQ(2,1, -1^3)$ and associated Siegel--Veech constants}\label{fig:tab21iii}
\end{figure}

Summing on all configurations we obtain:
\[c_{area}(\cQ(2, 1, -1^3))=\cfrac{49\pi^2}{24\Vol\cQ_1(2, 1, -1^3)}, \]
which is coherent with the values given in \S~\ref{volexample}.

\subsubsection{$\cQ(4, -1^4)$}
All configurations with cylinders for this stratum are given on Figure \ref{fig:tab4iiii}.
\begin{figure}[h!]
\centering
\includegraphics[scale=1]{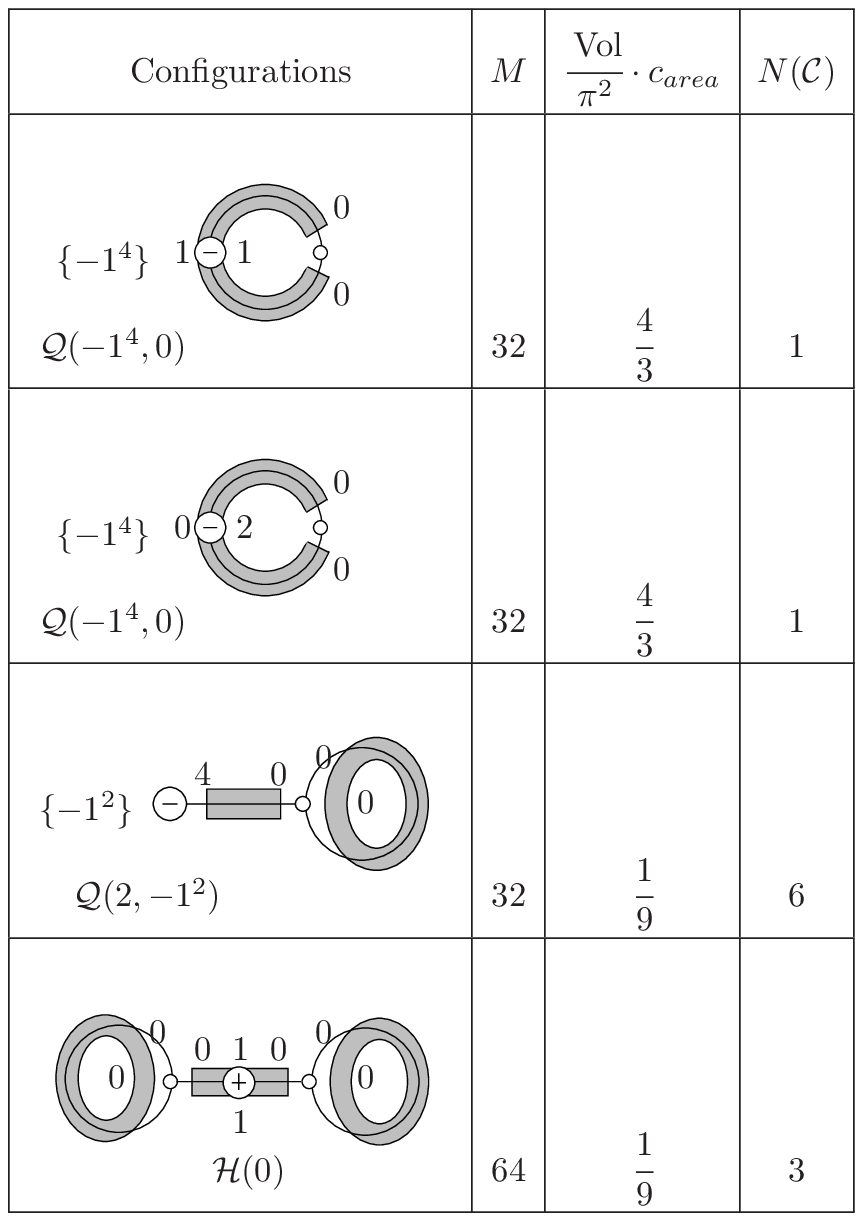}
\caption{Configurations containing cylinders for $\cQ(4, -1^4)$ and associated Siegel--Veech constants}\label{fig:tab4iiii}
\end{figure}

Summing on all configurations we obtain:
\[c_{area}(\cQ(4, -1^4))=\cfrac{11\pi^2}{3\Vol\cQ_1(4, -1^4)}, \]
which is coherent with \S~\ref{volexample}.

\subsubsection{$\cQ(4, 1, -1)$} 
All configurations with cylinders for this stratum are given in Figure \ref{fig:tab41i}.
\begin{figure}[h!]
\centering
\includegraphics[scale=1]{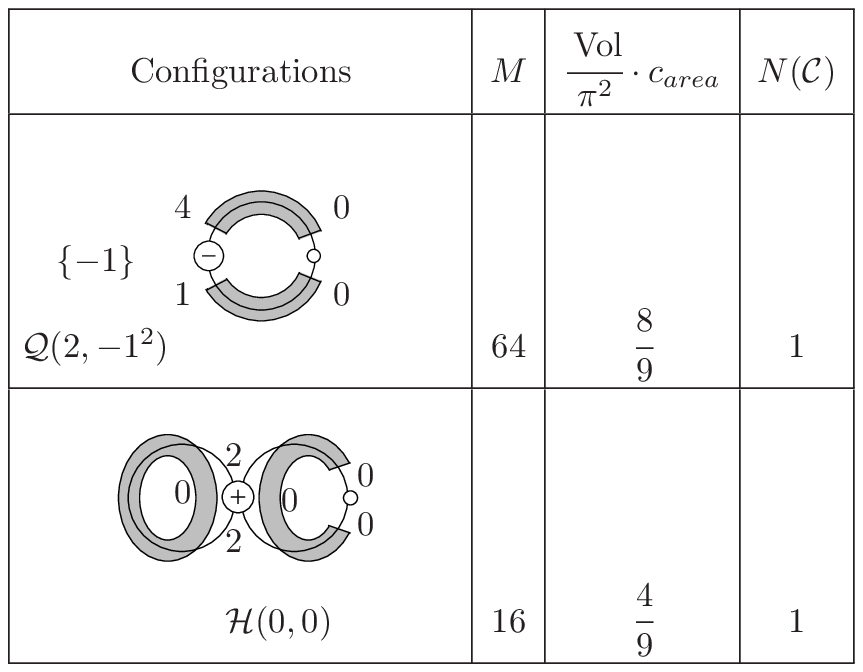}
\caption{Configurations containing cylinders for $\cQ(4,1, -1)$ and associated Siegel--Veech constants}\label{fig:tab41i}
\end{figure}

Summing on all configurations we obtain:
\[c_{area}(\cQ(4,1, -1))=\cfrac{4\pi^2}{3\Vol\cQ_1(4,1,-1)}, \]
which is coherent with \S~\ref{volexample}.

\subsubsection{$\cQ(3, 2, -1)$}
 All configurations with cylinders for this stratum are given in Figure \ref{fig:tab32i}.
\begin{figure}[h!]\begin{center}
\includegraphics[scale=1]{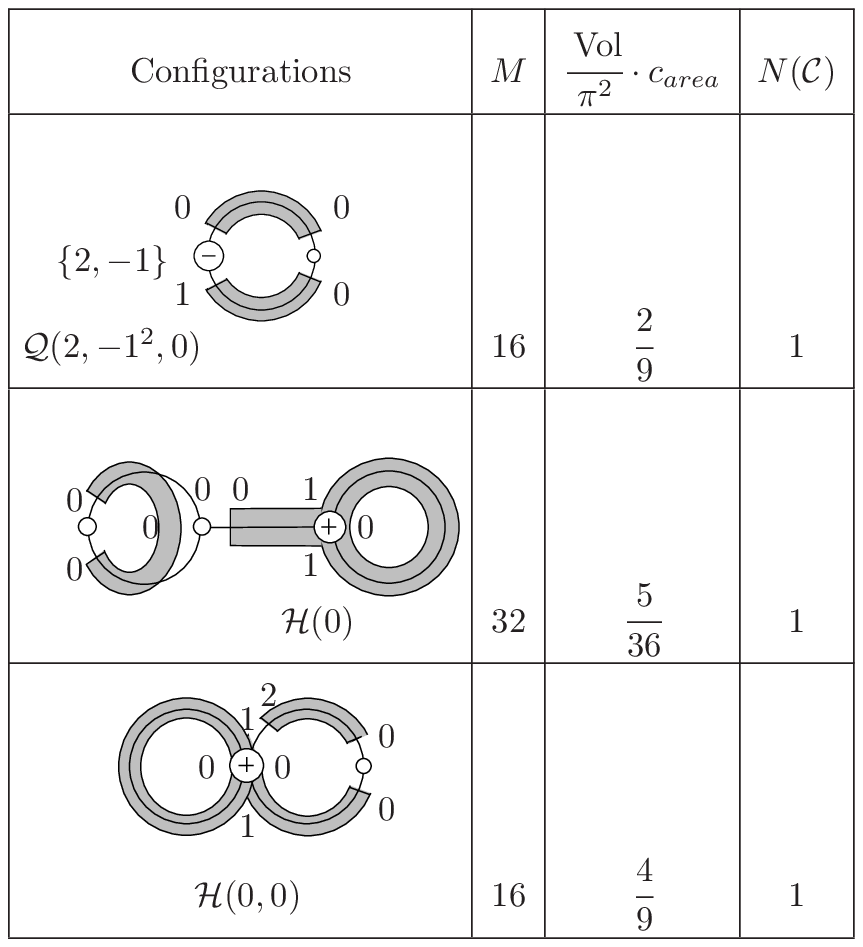}
\end{center}
\caption{Configurations containing cylinders for $\cQ(3,2, -1)$ and associated Siegel--Veech constants}\label{fig:tab32i}
\end{figure}

Summing on all configurations we obtain:
\[c_{area}(\cQ(3,2, -1))=\cfrac{29\pi^2}{36\Vol\cQ_1(3,2,-1)}, \]
which is coherent with \S~\ref{volexample}.

\subsubsection{$\cQ^{non}(6, -1^2)$}
This stratum has two connected components, one hyperelliptic, the other not. Since we have already studied the hyperelliptic component case in \S~\ref{sect:volhyp}, we consider only the remaining component. Admissible configurations for this components are obtained by taking off the hyperelliptic configurations from the list of all configurations for the stratum. We obtain the list presented in Figure \ref{fig:tab6ii}.
\begin{figure}[h!]
\centering
\includegraphics[scale=1]{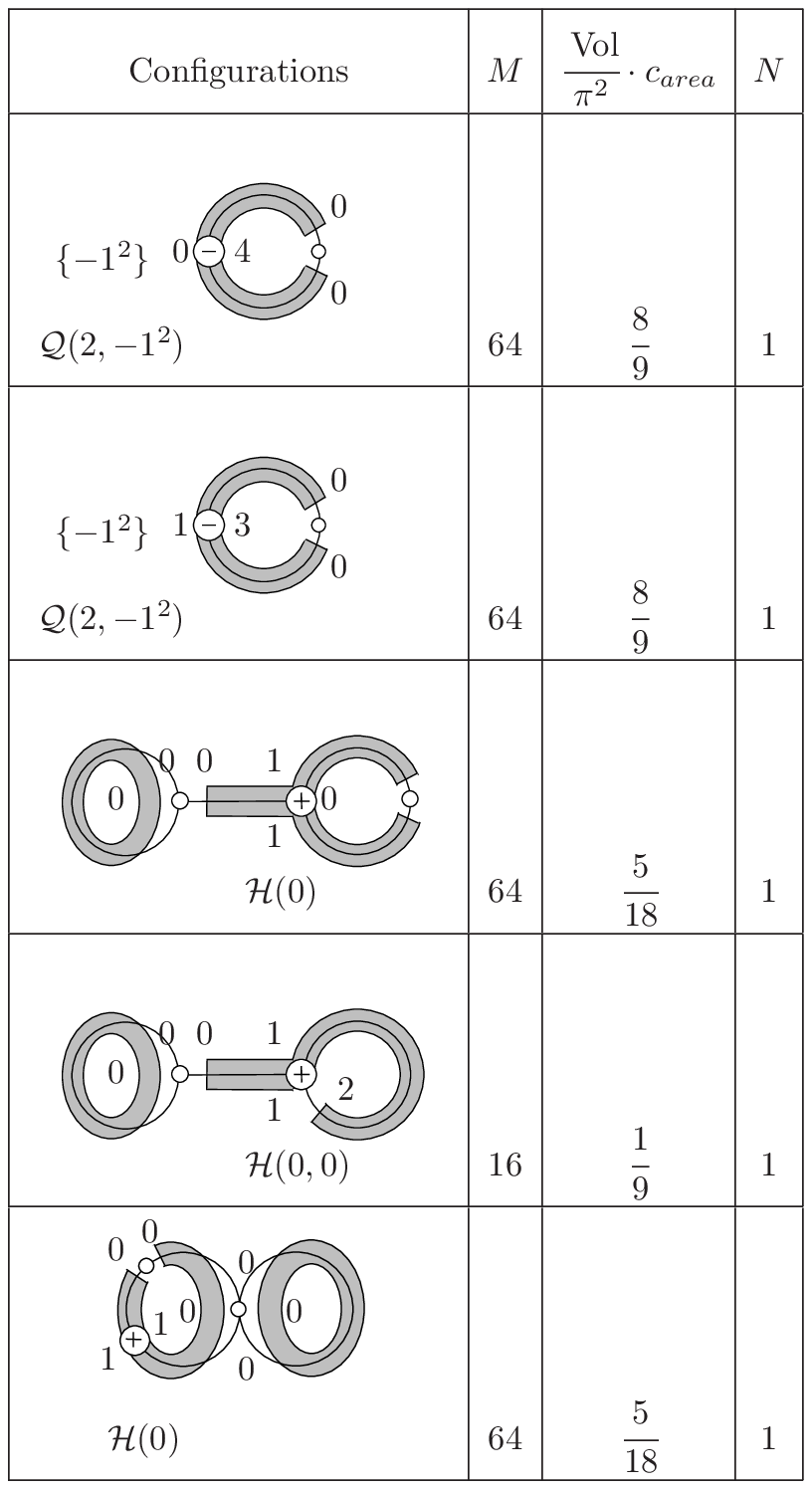}
\caption{Configurations containing cylinders for $\cQ^{non}(6, -1^2)$ and associated Siegel--Veech constants}\label{fig:tab6ii}
\end{figure}

Summing on all configurations we obtain:
\[c_{area}(\cQ^{non}(6, -1^2))=\cfrac{22\pi^2}{9\Vol\cQ_1^{non}(6,-1^2)}, \]
as expected.

\subsubsection{$\cQ(8)$}\label{Q8}
This stratum is non-varying, and its configurations (Figure \ref{fig:tab8}) present extra-symmetries.

\begin{figure}[h!]
\centering
\includegraphics[scale=1]{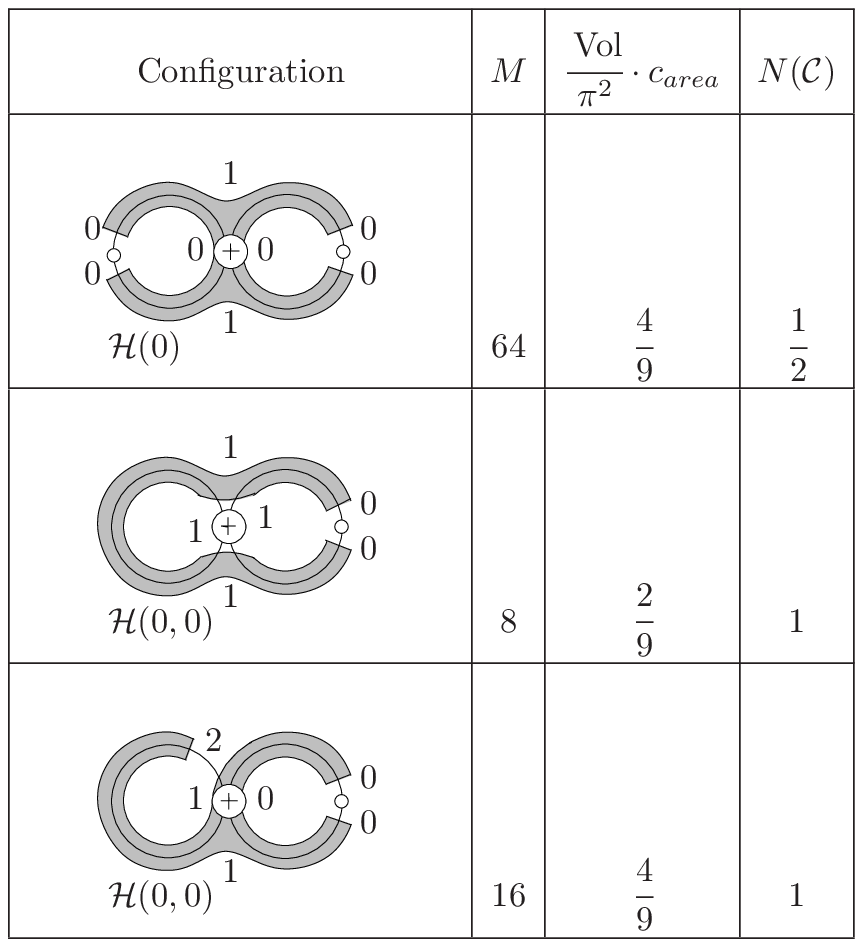}
\caption{Configurations containing cylinders for $\cQ(8)$ and associated Siegel--Veech constants}\label{fig:tab8}
\end{figure}

Summing on all configurations we obtain: 
\[c_{area}(\cQ(8))=\cfrac{8}{9}\cdot\cfrac{\pi^2}{\Vol},\]
as expected.

Note that the first configuration possesses a decorated ribbon graph symmetry that intertwines the two cylinders and stabilizes the boundary surface and the new born zero. That explains the factor $1/2$ for $N(\cC)$ (cf \S~\ref{sect:count}). 

\subsection{Dimension 6}
Here we treat only the varying strata, to obtain new values of Siegel--Veech constants. For the other strata, one can check that the computations are coherent using the values given in \S~\ref{volexample}. The only varying strata in dimension 6 are $\cQ(1^3, -1^3)$ and $\cQ(2^2, -1^4)$. Since the first one is principal and studied in \S~\ref{ss:exapp}, we detail configurations only for the second one.

\subsubsection{$\cQ(2^2, -1^4)$}

The configurations for this stratum are presented on Figure \ref{fig:tab22iiii}.
\begin{figure}[h!]
\centering
\includegraphics[scale=1]{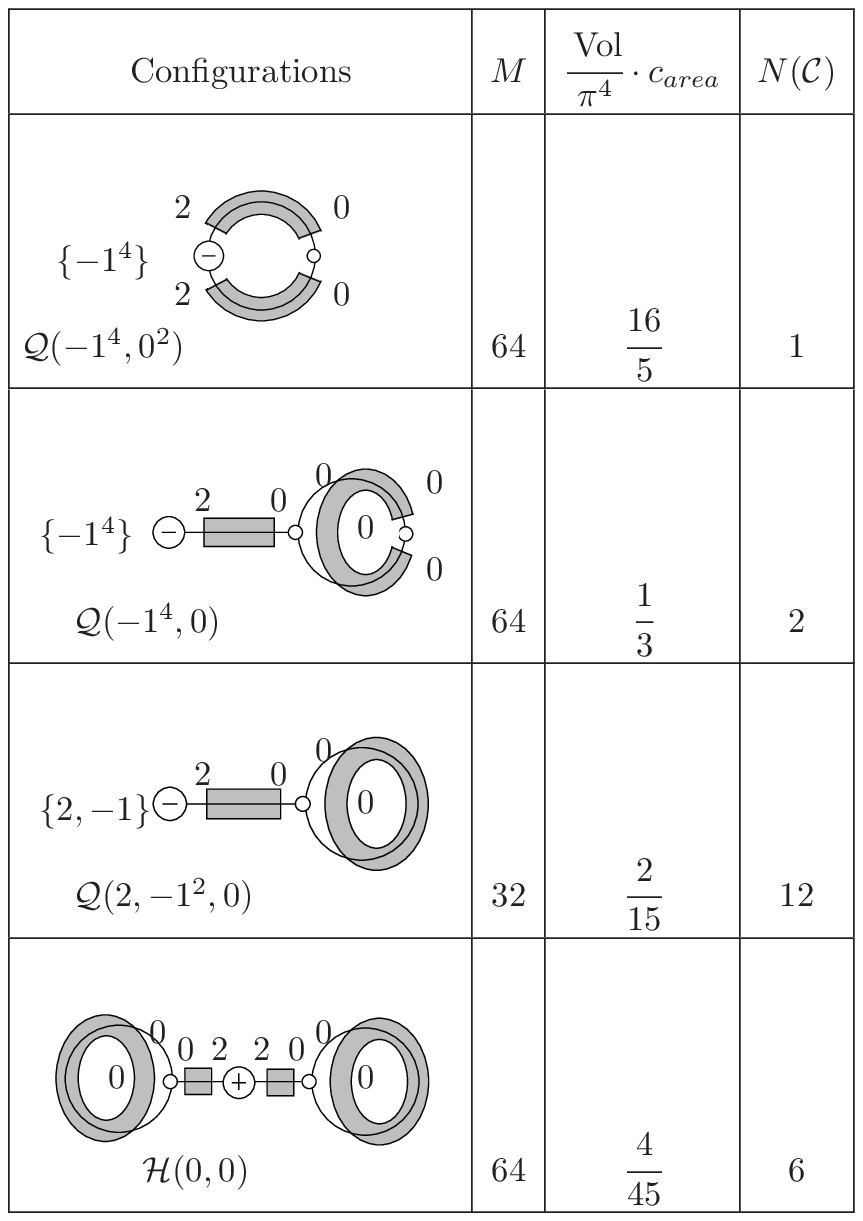}
\caption{Configurations containing cylinders for $\cQ(2^2, -1^4)$ and associated Siegel--Veech constants}\label{fig:tab22iiii}
\end{figure}

Summing on all configurations we obtain
\[c_{area}(\cQ(2^2, -1^4))=6\cfrac{\pi^4}{\Vol\cQ_1(2^2, -1^4)}=\cfrac{135}{68\cdot \pi^2}, \]
which corresponds with the approximate value coming from the Lyapunov exponents.

\appendix
\section{Geometry of configurations containing cylinders}\label{sect:geom}

This appendix develops the quadratic version of some geometric results on configurations, proven in the Abelian case in \cite{BG}.

\subsection{Variants of Siegel--Veech constants}
The result (\ref{eq:careaandc}) of Theorem \ref{th:ccyl} can be interpreted as follows: the ratio $ \cfrac{c_{area}(\cC)}{c_{cyl}(\cC)} $ represents the mean area of a cylinder in configuration $\cC$. It does not depend on the configuration, but only on the dimension of the stratum. Summing on all configurations in a stratum we get a result of Vorobets (Theorem 1.6 in \cite{Vo}).

We introduce variants of Siegel--Veech constants whose ratios admit a geometric interpretation. Some of them were introduced by Vorobets.

 We define $ N_{A_1\geq p}(S, \cC, L) $ (resp. $ N_{A\geq p}(S, \cC, L) $) that counts configurations $\cC$ on $S$ only if the area of a fixed cylinder (resp. all cylinders) fills at least a proportion $p$ of the area of the entire surface. As before we denote \[c_*(\cC)=\lim\limits_{L\to\infty}\frac{N_*(S, \cC, L)\cdot(\mbox{Area of }S)}{\pi L^2}\] the associated Siegel--Veech constants.

We give the analogue of Theorems
4 and 5 of \cite{BG}. Proofs are very similar to the Abelian case so we keep them short.

We introduce the {\it incomplete Beta function} \[B(x;n,q)=\int_0^{x}u^{n-1}(1-u)^{q-1}\dd u\] and the {\it Beta function} $B(n,q)=B(1; n,q)$. It is a standard fact that \[B(x; n,q)=B(n,q)\sum_{k=n}^ {n+q-1}{n+q-1\choose k}x^ k(1-x)^ {n+q-1-k}.\]

\begin{theorem}Let $\cC$ be an admissible configuration for a connected stratum $\cQ(\alpha)$ of quadratic differentials. Let $q$ denote the total number of cylinders. Assume that the boundary stratum $\cQ(\alpha')$ is non empty, and $q\geq 1$. Then the ratios of Siegel--Veech constants associated to $\cC$ are the following:
\begin{eqnarray}\cfrac{c_{A>p}(\cC)}{c(\cC)}& = & \frac{B(1-p; n_S,q)}{B(n_S,q)}\label{eq:cAp}\\
\cfrac{c_{A_1>p}(\cC)}{c(\cC)} & = & (1-p)^{\dim_\C\cQ(\alpha)-2}\label{eq:cA1p}\end{eqnarray}
\end{theorem}
The first ratio can be interpreted as the probability for the cylinders to fill a large part of the area of the surface, and the second ratio the probability for a distinguished cylinder to fill a large part of the area of the surface. Note that the first ratio depends on the number of cylinders $q$ in the configuration, as the second ratio depends only of the dimension of the stratum.

\begin{proof}
We begin with the proof of (\ref{eq:cAp}). We follow step by step the computations of \S~\ref{sssect:compc}. The value of $Cusp(\varepsilon)$ does not change. The only adjustment to make is that the area of the surface $ r_TT $ which we glue to $ r_SS' $ has to satisfy $r_T^2> p(r_T^2+r_S^2)$, which is equivalent to $r_T> \sqrt{\frac{p}{1-p}}r_S$. So (\ref{eq:nut}) becomes
\begin{equation*}\nu_{T}^{A>p}(\Omega(\varepsilon,r_S))=\frac{M_c\pi\varepsilon^2}{M_t2^{q}(q-1)!}\int_{\sqrt\frac{p}{1-p}r_S}^{\sqrt{1-r_S^2}}r_T^{2q-1}(r_S^2+r_T^2)\dd r_T.\end{equation*}
and using the constraint $ r_T^2+r_2^2\leq 1 $ we obtain the following bound on $r_S$: $r_S\leq \sqrt{1-p}$, so (\ref{eq:mu}) becomes 
\begin{equation*}\mu^{A>p}(C(\cQ_1^\varepsilon(\cC))) =\frac{M\Vol(\cQ_1(\alpha'))\pi\varepsilon^2}{2^{q+1}(q+1)!}\underbrace{\int_0^{\sqrt{1-p}} r_S^{2n_S-1}\int_{\sqrt\frac{p}{1-p}r_S}^{\sqrt{1-r_S^2}}r_T^{2q-1}(r_S^2+r_T^2)\dd r_T\dd r_S}_{I_p}
\end{equation*}
Using an appropriate change of variables as the Abelian case, we recognize
\[I_p=\frac{B(1-p; n_S,q)}{4(n_S+q+1)}\] where $B(1-p; n_S,q)$ is the incomplete Beta function. Comparing the result to (\ref{eq:cSV}) we get (\ref{eq:cAp}).

Now we compute $ c_{A_1>p}(\cC) $: we have the same constraints as before, plus the additional constraint that the first cylinder fills at least a part $p$ of the area of the surface. This affects the calculus of $Cusp$. Note that a cylinder in $S\in\cQ_1(\alpha)$ fills at least a part $p$ of the surface if it fills at least part $a=p\cdot \frac{r_S^2+r_T^2}{r_T^2}$ in the space of the cylinders $\cT_1$. So we have to replace $Cusp(\varepsilon)$ by \begin{multline*}Cusp^{A_1>a}(\varepsilon)=2(q+1)\nu_{T}^{A_1>p}(C(\cT_1^{\varepsilon}))\\=2(q+1)\frac{M_c}{M_t}\pi\int_{0}^{\frac{\epsilon}{2}}
w^{q+1}\dd w\int_{\R_{+}^{q}}  \chi{\left\{\cfrac{w}{2\varepsilon^2}\leq h'\leq \cfrac{1}{2w}\right\}}  \chi{\left\{h_1\geq ah'\right\}} \dd h_1\dots \dd h_q'.\end{multline*}
Using the change of variables $h_1'=h_1-ah$ we get:
\[Cusp^{A_1>a}(\varepsilon)=Cusp(\varepsilon)\cdot (1-a)^{q-1}\]
Note that if we choose another cylinder, the computations are exactly the same, even if it is a thick cylinder.
Finally (\ref{eq:mu}) becomes
\begin{multline*}\mu^{A_1>p}(C(\cQ_1^\varepsilon(\cC))) =\frac{M\Vol(\cQ_1(\alpha'))\pi\varepsilon^2}{2^{q+1}(q+1)!}\\ \cdot\underbrace{\int_0^{\sqrt{1-p}} r_S^{2n_S-1}\int_{\sqrt\frac{p}{1-p}r_S}^{\sqrt{1-r_S^2}}r_T^{2q-1}(r_S^2+r_T^2)\left(1-p\cdot\frac{r_S^2+r_T^2}{r_T^2}\right)\dd r_T\dd r_S}_{I_p'}
\end{multline*}
and we get \[I'_p=\frac{(1-p)^{n_S+q-1}}{4(n_S+q+1)}\cdot B(n,q).\]
Comparing the result to (\ref{eq:cSV}) we get (\ref{eq:cA1p}).
\end{proof}

\subsection{Maximal number of cylinders}\label{ssect:maxcyl}

Configurations of quadratic differentials in genus $0$ are detailed in \cite{AEZ}. They contain at most one cylinder. The following proposition gives the maximal number of cylinders in a configuration in higher genus.

\begin{proposition}\label{prop:nbermaxcylquad}
Consider a stratum $Q(\alpha)$ in genus $g\geq 1$, with $\alpha=(4l_1, \dots, 4l_m, 4k_1+2, \dots, 4k_n+2, b_1, \dots, b_p, -1^k)$, and $l_i\geq 0$, $k_i\geq 0$, $b_i$ odd. 
First assume that $2n+\displaystyle{\sum_{i=1}^{p}}b_i -k+4\geq 0$, then the maximal number of \^homologous cylinders satisfies:
$$ q_{max}(\alpha)=\lfloor\frac{n}{2}\rfloor+m+\epsilon_{\alpha},$$ where $\epsilon_{\alpha}\in\{0,1,2\}$.

Without this assumption, the maximal number of \^homologous cylinders is given by:
$$\begin{array}{ccc}q_{max}(\alpha)&=&\max\{\card I+\cfrac{\card J}{2};\; I\subset \{1, \dots, m\}, J\subset \{1,\dots n\}, \card J \;\textrm{even},\\&& 4\displaystyle\sum_{i\in I}l_i+4\sum_{j\in J}k_j+2n+\sum_{k=1}^{p}b_k+4-k\geq 0\}+\epsilon_{\alpha}\end{array}$$

\end{proposition}

To prove this proposition we will need the following lemma:
\begin{lemma}\label{lem:odd:zeros}Odd zeros are created by surfaces of non trivial holonomy $\ominus$ or by loops in the graphs of configurations. At most four newborn odd zeros can be created in a configuration.
\end{lemma}

\begin{proof} Since the zeros on which we perform surgeries on surfaces of non trivial holonomy $\ominus$ are of any order (even or odd), it is easy to see that we can obtain any parity order for newborns zeros created by surfaces $\ominus$. 

This is not the case of surfaces $\oplus$. In fact, a newborn zero represented in the graph by a boundary of a ribbon graph which frames a chain of surfaces $\oplus$ (as in the picture below) surrounded by surfaces $\oplus$ or cylinders has always an even order. This is due to the fact that we perform surgeries such as creating a hole on surfaces of trivial holonomy, so on singularities of cone angle $2k\pi$. If we glue all these surfaces identifying all boundary singularities, then the new cone angle is also multiple of $2\pi$, so the newborn zero is of even order. Boundary types involved in these chains are $\circ 2.2$, $+2.1$, $+2.2$, $+3.2a$, $+3.2b$, $+3.3$, $+4.2a$, $+4.3a$, $+4.4$.
\begin{figure}[h!]
\centering
\includegraphics[scale=0.4]{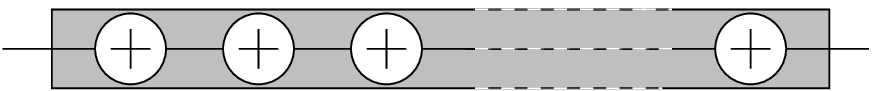}
\caption{Chain of surfaces $\oplus$}
\end{figure}

Then we just have to look at the remaining cases, namely, graphs containing surfaces of boundary type $\circ 3.2$ $\circ 4.2$, $+3.1$, $+4.1a$, $+4.1b$, $+4.2b$, $+4.2c$, $+4.3b$. Then one can see case by case that if the ribbon graph is locally as on the picture above, one or two odd zeros are created (one can replace the surface $\oplus$ by a cylinder $\circ$).
\begin{figure}[h!]
\centering
\includegraphics[scale=0.4]{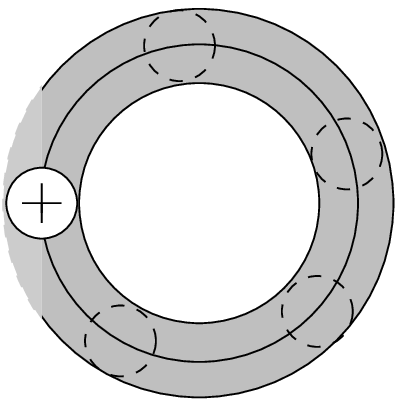}
\caption{Loop in the graph of configuration}
\end{figure}

As an example, Figure \ref{fig:creatingpole} represents how poles are created by loops in the graph of the configuration. \end{proof}

\begin{proof}[Proof of Proposition \ref{prop:nbermaxcylquad}]
This result is a corollary of the classification of configurations of \^homologous cylinders by Masur and Zorich (Figure \ref{fig:fig3MZ}). Each configuration is represented by a graph with one, two or three chains of surfaces $\oplus$ (with trivial linear holonomy) and cylinders $\circ$ (see also \S~\ref{ssect:graph} for more details about these graphs).
Then there are some remarks:
\begin{itemize}
\item A surface $\oplus$ of type $+2.1$ (cf Figure \ref{fig:fig6MZ}) in a chain is surrounded by at most two cylinders. In that case if there is no interior singularity it creates a newborn zero of order $4g=k_1+k_2+2$, where $g$ is the genus of the boundary strata $\cH\left(\frac{k_1+k_2-2}{2}\right)$ ($k_1$ and $k_2$ are odd).
\begin{figure}[h!]
\centering
\includegraphics[scale=1]{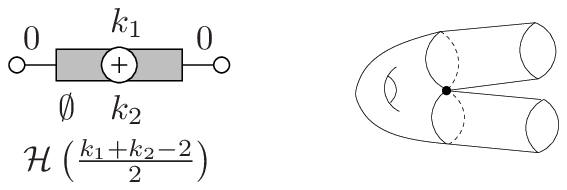}
\caption{Creating a zero of order $4g$}
\end{figure}
\item A surface $\oplus$ of type $+2.2$ in a chain is surrounded by at most two cylinders and in that case if there is no interior singularity it creates two newborn zeros of order $k_1$ and $k_2$ (even) with $k_1+k_2=4g$ where $g$ is the genus of the boundary strata $\cH\left(\frac{k_1-2}{2}, \frac{k_2-2}{2}\right)$.
\begin{figure}[h!]
\centering
\includegraphics[scale=1]{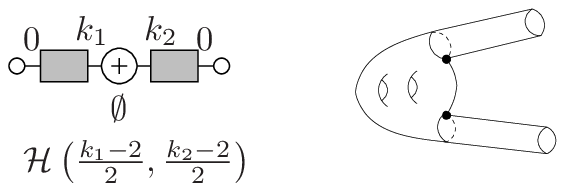}
\caption{Creating two zeros}
\end{figure}
\item By Lemma \ref{lem:odd:zeros}, at most four zeros of odd order can be realized as newborn zeros (created by loops in the graph of the configuration or by surfaces $\ominus$), the others are necessarily interior singularities (of surfaces $\ominus$).
\item Realizing zeros as newborn zeros instead of interior singularities increases the number of cylinders. 
\end{itemize}
First we assume that $2n+\sum_{i=1}^{p}b_i -k+4\geq 0$. One procedure to construct the configuration containing the most cylinders is the following: we consider all zeros of order $4l$ and realize them as newborn zeros with a surface of type $+2.1$ as described above. Then we consider the other even zeros and realize them by pairs as newborn zeros with surfaces of type $+2.2$ as described above. At this stage we obtain a chain of $m+\lfloor \frac{n}{2}\rfloor$ surfaces with a cylinder between each surface $\oplus$. We consider the remaining zeros (at most one even zero and all the odd zeros). If there are at least five odd zeros, we have to choose graph $a)$, $b)$ or $c)$ following notations of Figure \ref{fig:fig6MZ} to complete your configuration. If not, we can choose graph $c)$, $d)$ or $e)$. In all cases we will get at most $2$ additional cylinders, by looking carefully at all possible configurations depending on the number of odd/even zeros and poles.

In the general case, we have to choose carefully the even zeros that we realize as newborn zeros. Indeed all remaining zeros should be produced by another surface of non-negative genus in a boundary strata. This condition implies that we can choose to realize zeros of orders $4l_i$ or pairs of zeros $4k_{j_1}+2$, $4k_{j_2}+2$ with $i\in I$ and $j_1,j_2\in J$ while $4\sum_{i\in I}l_i+4\sum_{j\in J}k_j+2n+\sum_{k=1}^{p}b_k+4-k\geq 0$. This explains the general formula for the maximal number of cylinders.
\end{proof}

We are interested in the asymptotic geometry of configurations, in particular when the genus or the number of zeros tends to infinity, so we will consider $\tilde q_{max}(\alpha)=q_{max}(\alpha)-\epsilon_{\alpha}$ instead of $q_{max}(\alpha)$, to simplify the computations.

As a corollary of Proposition \ref{prop:nbermaxcylquad} we obtain that the strata maximizing the number of cylinders at genus fixed are the ones with the most even zeros:
\begin{corollary}Fix the genus $g\geq 1$ and the number of poles $k$. Denote $\Pi(4g-4+k)$ the set of partitions $\alpha$ of $4g-4+k$, and $l=\lfloor\frac{k}{4}\rfloor$. Then: 
$$\underset{\alpha\in\Pi(4g-4+k)}{\max} \tilde{q}_{max}(\alpha\cup\{-1^k\})=g+l-1$$ and the maximum is realized for $\alpha\in \Pi k' \sqcup \Pi_{4,2}(4g-4+4l)$, where $k=4l+k'$ and $\Pi_{4,2}(4g-4+4l)$ denotes the set of partitions of $4g-4+4l$ using only $4$ and $2$.
\end{corollary}

Recall that the mean area of a cylinder is given by $ \cfrac{1}{\dim_\C\cQ(\alpha)-1} $ (cf Theorem \ref{th:ccyl}), so $\cfrac{q_{max}(\alpha)}{\dim_{\C}(\cQ(\alpha))-1}$ represents the maximum mean total area of the cylinders in stratum $\cQ(\alpha)$. As another corollary of Proposition \ref{prop:nbermaxcylquad}, we obtain Proposition~\ref{prop:qmax:dim}.

\subsection{Configurations with simple surfaces} 

This section provides an answer in the quadratic case to the following question of Alex Eskin and Alex Wright: for a given stratum or a connected component of a stratum is it possible to find an admissible configuration
whose boundary surfaces are only tori ?

Lemma \ref{lem:odd:zeros} gives the main obstruction to solve this problem in the quadratic case: odd zeros are created by surfaces of non trivial holonomy $\ominus$ or by loops in graphs of configurations, and there are at most two surfaces of non trivial holonomy or two loops in a configuration. That means that a stratum with enough odd zeros will never have a configuration with only tori as boundary surfaces.

The second obstruction is that, as in the case of Abelian differentials, there is no way to have a decomposition into simple surfaces in hyperelliptic components of strata, since they are made from at most two surfaces and cylinders (cf \cite{B} and \S~\ref{sect:volhyp}).

Considering these two obstructions (odd zeros and hyperelliptic components), we can formulate the following result, which is very similar to the case of Abelian differentials \cite{BG}.

\begin{proposition}Let $\cQ^{comp}(\alpha_1, \dots, \alpha_s)$ be a connected component of a stratum of quadratic differentials, which is not hyperelliptic. If all the $\alpha_i$ are even then there exists a configuration in this component containing only tori and cylinders.
\end{proposition}

\begin{proof}Denote $n$ the number of zeros of order $4k+2$ and $m$ the number of zeros of order $4k$. As in the case of Abelian differentials, we just look at what type of zeros can be created by chains of tori and cylinders. We obtain the same type of zeros as in the case of Abelian differentials.

For the first type represented in the picture above, the cone angle around the singularity is also $2(2k+1)\pi$ so we obtain a zero of order $4k$.

\begin{figure}[h!]
\centering
\includegraphics[scale=1]{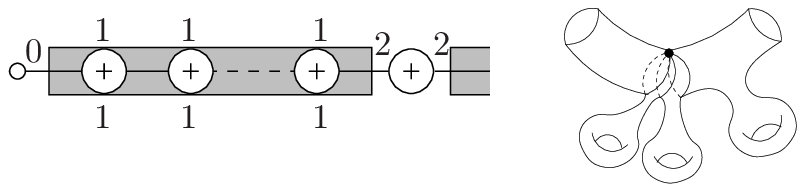}

\caption{Zero of the first type}
\end{figure}

Zeros of the second type represented above have order $4k+2$ since the cone angle is $(4k+4)\pi$.

\begin{figure}[h!]
\centering
\includegraphics[scale=1]{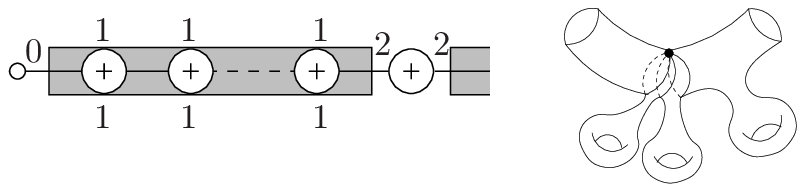}

\caption{Zero of the second type}
\end{figure}

Finally zeros of the third type are of order $4k+4$.

\begin{figure}[h!]
\centering
\includegraphics[scale=1]{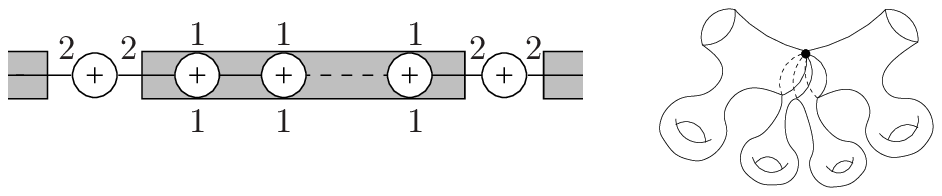}
\caption{Zero of the third type}

\end{figure}

With these chains we can easily construct a bigger chain that realizes all zeros. It remains to embed this chain in a graph of configuration. We can see that if there is at least two zeros of order greater than $4$, or if there is at least one zero of order greater than $8$, then we can embed this chain in the graph $e)$ with local ribbon graph of type $+4.2a$.

Since $\cQ(4)$ is empty, it remains only strata $\cQ(2,2,\dots, 2)$, which is realizable with a graph of type $e)$ and a local ribbon graph of type $\circ4.2$, by example.

\end{proof}


\bibliographystyle{amsalpha}

\end{document}